\documentclass[12pt,reqno]{amsart}
\usepackage{amsmath, amsthm, amssymb}
\usepackage{hyperref}
\usepackage{enumerate}
\usepackage{url}

\let\OLDthebibliography\thebibliography
\renewcommand\thebibliography[1]{
  \OLDthebibliography{#1}
  \setlength{\parskip}{0pt}
  \setlength{\itemsep}{0pt plus 0.3ex}
}

\usepackage{mathtools}

\usepackage{multirow, array}
\usepackage{placeins}

\usepackage{caption} 
\advance \topmargin by -\headheight
\advance \topmargin by -\headsep
     
\evensidemargin \oddsidemargin
\newtheorem{thm}{Theorem}[section]
\newtheorem{lemma}[thm]{Lemma}

\newtheorem{cor}[thm]{Corollary}

\theoremstyle{definition}

\theoremstyle{remark}

\numberwithin{equation}{section}

\newcommand{\mmod}[1]{\,\,\text{mod}\,\,#1}

\newcommand*\wrapletters[1]{\wr@pletters#1\@nil}
\def\wr@pletters#1#2\@nil{#1\allowbreak\if&#2&\else\wr@pletters#2\@nil\fi}

\def\alp{{\alpha}} 
\def\bet{{\beta}}  
\def\gam{{\gamma}} 
\def\del{{\delta}} \def\Del{{\Delta}}

\def\tet{{\theta}}  
\def\kap{{\kappa}}
\def\lam{{\lambda}} \def\Lam{{\Lambda}}

\def\sig{{\sigma}}

\def\om{{\omega}} \def \Om {{\Omega}}
\def\ome{{\omega}}  
\def\eps{\varepsilon}

\def\le{\leqslant} \def\ge{\geqslant}

\def\d{{\,{\rm d}}}

\def \sig{{\sigma}}

\def \bC {\mathbb C}

\def \bK {\mathbb K}
\def \bN {\mathbb N}

\def \bQ {\mathbb Q}
\def \bR {\mathbb R}
\def \bZ {\mathbb Z}

\def \ba {\mathbf a}
\def \bb {\mathbf b}

\def \bf {\mathbf f}

\def \bi {\mathbf i}
\def \bj {\mathbf j}

\def \bm {\mathbf m}
\def \bt {\mathbf t}
\def \bu {\mathbf u}
\def \bv {\mathbf v}
\def \bx {\mathbf x}

\def \by {\mathbf y}
\def \bz {\mathbf z}

\def \bzero {\mathbf 0}

\def \btau {\boldsymbol{\tau}}
\def \bmu {{\boldsymbol{\mu}}}
\def \bnu {{\boldsymbol{\nu}}}
\def \balp {\boldsymbol{\alp}}
\def \bbet {\boldsymbol{\beta}}

\def \bgam {\boldsymbol{\gam}}
\def \bxi {{\boldsymbol{\xi}}}

\def \bomega {\boldsymbol{\omega}}
\def \bome {\boldsymbol{\omega}}
\def \bom {\boldsymbol{\omega}}

\def \fm {\mathfrak m}

\def \ft {\mathfrak t}

\def \fM {\mathfrak M}
\def \fN {\mathfrak N}

\def \fQ {\mathfrak Q}
\def \fR {\mathfrak R}

\def \fU {\mathfrak U}

\def \cD {\mathcal D}

\def \cF {\mathcal F}

\def \cI {\mathcal I}

\def \cS {\mathcal S}

\def \cX {\mathcal X}

\def \rank {\mathrm{rank}}

\def \det {\mathrm{det}}

\def \sinc {\mathrm{sinc}}

\def \diag {\mathrm{diag}}
\def \dag {\dagger}
\def \diam {\diamond}

\begin{document}
\title[Birch's theorem with shifts]{Birch's theorem with shifts}
\author[Sam Chow]{Sam Chow}
\address{School of Mathematics, University of Bristol, University Walk, Clifton, Bristol BS8 1TW, United Kingdom}
\email{Sam.Chow@bristol.ac.uk}
\thanks{}
\date{}
\begin{abstract} 
A famous result due to Birch (1961) provides an asymptotic formula for the number of integer points in an expanding box at which given rational forms of the same degree simultaneously vanish, subject to a geometric condition. We present the first inequalities analogue of Birch's theorem.
\end{abstract}
\maketitle

\section{Introduction}
\label{intro}

A famous result due to Birch \cite[Theorem 1]{Bir1962} provides an asymptotic formula for the number of integer points in an expanding box at which given rational forms of the same degree simultaneously vanish, subject to a geometric condition. This in particular implies the existence of a nontrivial solution to the system of homogeneous equations, providing that a nonsingular solution exists in every completion of the rationals. We present the following inequalities analogue of Birch's theorem.

\begin{thm} \label{MainThm}
Let $f_1, \ldots, f_R$ be rational forms of degree $d \ge 2$ in 
\begin{equation} \label{numvars}
n > \sig + R(R+1)(d-1)2^{d-1}
\end{equation}
variables, where $\sigma$ is the dimension of the affine variety cut out by the condition
\[
\rank(\nabla f_k)_{k=1}^R < R.
\]
Assume that the forms $(1,\ldots,1) \cdot \nabla f_k$ $(1 \le k \le R)$ are linearly independent. Let $\btau \in \bR^R$, and let $\eta$ be a positive real number. Let $\mu$ be an irrational real number, and write $\bmu = (\mu,\ldots,\mu) \in \bR^n$. Then the number \mbox{$N(P) = N_\bf(P;  \mu, \btau, \eta)$} of integer solutions $\bx \in [-P,P]^n$ to
\begin{equation*}
|f_k(\bx + \bmu) - \tau_k| < \eta \qquad (1 \le k \le R)
\end{equation*}
satisfies
\begin{equation} \label{asymp}
N(P) = (2\eta)^R c P^{n-Rd} + o(P^{n-Rd})
\end{equation}
as $P \to \infty$, where 
\begin{equation} \label{cdef}
c = c_\bf = \int_{\bR^R} \int_{[-1,1]^n} e(\bgam \cdot \bf(\bt)) \d \bt \d \bgam.
\end{equation}
If $\bf = \bzero$ has a nonsingular real solution then $c > 0$.
\end{thm}

The definition \eqref{cdef} of the singular integral $c$ is the one given by Birch \cite{Bir1962}. We can interpret $c$ as the real density of points on the variety $\bf = \bzero$; we defer an extended discussion until \S \ref{SingularIntegral}.

Theorem \ref{MainThm} implies that $\{ \bf(\bx + \bmu) : \bx \in \bZ^n \}$ is dense in $\bR^R$. The example with $R = 1$ and
\[
f_1(\bx) = (x_1 - x_2)^3 + \ldots + (x_{99} - x_{100})^3
\]
shows that some condition, such as the linear independence of the forms $(1,\ldots,1) \cdot \nabla f_k$ $(1 \le k \le R)$, is necessary in order for our statement to be true. Theorem \ref{MainThm} involves a `uniform' shift $\bmu = (\mu, \ldots, \mu) \in \bR^n$. From our method it is not clear how to handle an arbitrary shift $\bmu = (\mu_1, \ldots, \mu_n) \in \bR^n \setminus \bQ^n$, as in \cite{scs,wps}, since many more simultaneous rational approximations would then be necessary.

In Theorem \ref{MainThm}, we have used the `Birch singular locus' to control the degeneracy of the system $\bf$. An alternative approach involves Schmidt's $h$-invariant \cite[\S 1]{Sch1985}. Quoting Schmidt, the $h$-invariant of a form $F$ of degree $d \ge 2$ with rational coefficients is the least $h$ such that $F$ `splits into $h$ products', i.e. 
\[
F = A_1B_1 + \ldots + A_hB_h
\]
for some forms $A_i$, $B_i$ of positive degrees and rational coefficients. If $f_1, \ldots, f_R$ are forms in $n$ variables with rational coefficients and have the same degree $d \ge 2$, then the $h$-invariant of the system $\bf$ is the minimum $h$-invariant of any form in the rational pencil. Writing $h$ for the $h$-invariant of $\bf$, we note that $h \le n$. Define $\Phi(d)$ by $\Phi(2) = \Phi(3) = 1$, $\Phi(4) = 3$, $\Phi(5) = 13$ and
\[
\Phi(d) = \frac{d!}{(\log 2)^d} \qquad (d \ge 6).
\]

\begin{thm} \label{TheoremTwo}
We may replace the condition \eqref{numvars} in Theorem \ref{MainThm} by the hypothesis
\begin{equation} \label{hhyp}
\frac h{\Phi(d)} > R(R+1)(d-1)2^{d-1} + R(R-1)(d-1),
\end{equation}
and the same conclusions hold.
\end{thm}

Cognoscenti will recall that in Schmidt's work \cite{Sch1985} the $h$-invariant needs to be larger if one seeks to ensure positivity of the singular series. This is not necessary for us: there is no singular series, since the main term comes from a single major arc around $\bzero$.

Over its half century of fame, Birch's theorem has been an extremely popular result to improve and generalise. In fact it may be possible for one to incorporate into Theorem \ref{MainThm} a recent improvement in Birch's theorem due independently to Dietmann \cite{Die2014} and Schindler \cite{Sch2014}. Skinner \cite{Ski1997} generalised Birch's theorem to number fields, and Lee \cite{Lee2012} considered Birch's theorem in a function field setting. Other results related to Birch's theorem are too numerous to honestly describe in a confined space, but recent papers include those of Brandes, Browning, Dietmann, Heath-Brown and Prendiville \cite{Bra2014, BDHB2014, BHB2014, BP2014}.

The case where $R=1$ and $f_1$ is an indefinite quadratic form has been solved in five variables by Margulis and Mohammadi \cite{MM2011}, who generalised famous results due to G\"otze \cite{Goe2004}, Margulis \cite{Mar1989} and others; four variables suffice unless the signature is $(2,2)$, while three variables suffice to obtain a lower bound of the expected strength. This present paper is a sequel to \cite{scs, wps}. The author was initially motivated to study shifted forms by Marklof's papers \cite{Mar2002, Mar2003}, which dealt with shifted quadratic forms in relation to the Berry--Tabor conjecture from quantum chaos; see also \cite{Mar2002cor}.

To our knowledge, no author has previously considered inhomogeneous diophantine inequalities of degree three or higher without assuming any additive structure, although inhomogeneous cubic equations were investigated by Davenport and Lewis \cite{DL1964}. For previous results on additive inhomogeneous diophantine inequalities see \cite{scs, wps}, where the author built on work of Freeman \cite{Fre2003}, who applied important estimates due to Baker \cite{Bak1986}. Some of these ideas were used by Parsell to treat simultaneous diagonal inequalities in \cite{Par1, Par2, Par3}. For homogeneous diophantine inequalities without additive structure, there is Schmidt's general result \cite[Theorem 1]{Sch1980}, as well as improved treatments of the cubic scenario due to Pitman \cite{Pit1968} and then Freeman \cite{Fre2000}. The more specialised cases of split cubic forms and cubic forms involving a norm form have been studied by the author \cite{sf} and Harvey \cite{Har2011}, respectively.

We now outline our proof of the asymptotic formula \eqref{asymp} in Theorem \ref{MainThm}. Our main weapon is Freeman's variant \cite{Fre2002} of the Davenport--Heilbronn method \cite{DH1946}. We may assume that the coefficients of $f_1,\ldots,f_R$ are integer multiples of $d!$. Indeed, we may if necessary rescale $\bf, \btau, \eta$, and change variables in the outer integral of \eqref{cdef}. Our starting point is the Taylor expansion
\begin{equation} \label{Taylor}
f_k(\bx + \bmu) = f_k(\bx) + f_k(\bmu) + \sum_{j=1}^{d-1} \mu^{d-j} \sum_{|\bj|_1 = j} d_{k,\bj} \bx^\bj \qquad (1 \le k \le R)
\end{equation}
about $\bmu$, where for $\bj \in \bZ_{\ge 0}^n$ we write
\[
\bx^\bj = x_1^{j_1} \cdots x_n^{j_n}, \quad |\bj|_1 = j_1+\ldots+j_n, \quad \bj! = j_1! \cdots j_n!
\]
and 
\begin{equation} \label{partials}
d_{k,\bj} = \bj !^{-1} \partial^\bj f_k(1, \ldots, 1) \in \bZ \qquad (1 \le k \le R).
\end{equation}
Thus, we may regard our shifted forms as polynomials in $\bx$. Note that
\begin{equation} \label{fknote}
f_k(\bx) = \sum_{|\bj|_1 = d} d_{k,\bj} \bx^\bj \qquad (1 \le k \le R).
\end{equation}

The pertinent exponential sums are
\begin{equation*}
S(\balp) = \sum_{|\bx| \le P} e(\balp \cdot \bf (\bx + \bmu)) \qquad (\balp \in \bR^R).
\end{equation*}
From \eqref{Taylor} we see that the highest degree component of $f_k(\bx + \bmu)$ is precisely $f_k(\bx)$. We can therefore use Birch's argument \cite{Bir1962}, which is based on Weyl differencing and the geometry of $\bf$, to restrict consideration to a thin set of major arcs where $\balp$ is well approximated. 

Though the polynomials $f_k(\bx + \bmu)$ are of the particular shape \eqref{Taylor}, we shall also need some exponential sum bounds in a more general inhomogeneous context. There are
\begin{equation} \label{NjDef}
N_j := {j+n-1 \choose n-1}
\end{equation}
monomials of degree $j$ in $n$ variables, or in other words there are $N_j$ vectors $\bj \in \bZ_{\ge 0}^n$ such that $|\bj|_1 = j$. For $\balp \in \bR^R$ and 
\begin{equation} \label{omegaDiam}
\bomega_\diam = (\omega_\bj)_{1 \le |\bj|_1 \le d-1} \in \bR^{N_1 + \ldots + N_{d-1}},
\end{equation}
write
\begin{equation} \label{gdef}
g(\balp, \bomega_\diam) = \sum_{|\bx| \le P} e\Bigl( \balp \cdot \bf(\bx) + \sum_{1 \le |\bj|_1 \le d-1} \omega_\bj \bx^\bj \Bigr).
\end{equation}
Using \eqref{Taylor}, we shall view $S(\balp)$ as a special case of $g(\balp, \bome_\diam)$, up to multiplication by a constant of absolute value 1. Thanks to the early steps of our argument, this will allow us to focus on the situation in which
\begin{equation} \label{gbig}
|g(\balp, \bome_\diam)| \ge P^n H^{-1},
\end{equation}
where $H$ is at most a small power of $P$.

Let
\begin{equation} \label{Fdef}
F_{k,j}(\bx) = \sum_{|\bj|_1 = j} d_{k,\bj} \bx^\bj \qquad (1 \le k \le R, 1 \le j \le d),
\end{equation}
and note from \eqref{fknote} that
\begin{equation} \label{top}
F_{k,d} = f_k \qquad (1 \le k \le R).
\end{equation}
We now see from \eqref{Taylor} that
\begin{equation} \label{Taylor2}
f_k(\bx+\bmu) = f_k(\bmu) + \sum_{1 \le j \le d} \mu^{d-j} F_{k,j}(\bx) \qquad (1 \le k \le R)
\end{equation}
and
\[
\balp \cdot \bf(\bx+ \bmu) = \balp \cdot \bf(\bmu) +  \balp \cdot \bf(\bx) 
+ \sum_{k \le R}\alp_k \sum_{1 \le |\bj|_1 \le d-1} \mu^{d-|\bj|_1} d_{k,\bj} \bx^\bj.
\]
Thus, with \eqref{omegaDiam} and the specialisation
\begin{equation} \label{omegadef}
\omega_\bj := \sum_{k \le R} d_{k,\bj} \alp_k \mu^{d-|\bj|_1} \qquad (1 \le |\bj|_1 \le d),
\end{equation}
we have
\begin{equation} \label{Sg}
S(\balp) = e(\balp \cdot \bf(\bmu)) g(\balp, \bome_\diam).
\end{equation}

Throughout, we define $\om_\bj$ $(|\bj|_1 = d)$ in terms of $\balp$ by
\begin{equation} \label{omegatop}
\om_\bj = \sum_{k \le R} d_{k,\bj} \alp_k \qquad (|\bj|_1 = d).
\end{equation}
This is consistent with \eqref{omegadef}. Though $\bom_\diam$ does not depend on those $\om_\bj$ for which $|\bj|_1 = d$, it will be convenient to also consider them.

Ideally, we would like to have good rational approximations to $\alp_k \mu^{d-j}$ for all $k \in \{1,2,\ldots,R \}$ and all $j \in \{ 1,2,\ldots,d \}$. We could then use the procedure demonstrated in \cite[ch. 8]{Bro2009} to decompose $S(\balp)$ into archimedean and non-archimedean components. We are only able to achieve this ideal for $j \in \cS$, where $\cS$ is the set of $j \in \{1,2,\ldots,d\}$ such that $F_{1,j}, F_{2,j}, \ldots, F_{R,j}$ are linearly independent. For all $j \in \{1,2,\ldots, d\}$, we are nonetheless able to rationally approximate those linear combinations of $\alp_1 \mu^{d-j}, \ldots, \alp_R \mu^{d-j}$ that are needed at this stage of the argument, namely the $\ome_\bj$. 

These rational approximations are a nontrivial consequence of \eqref{gbig}. The key idea is to fix all but one of the variables, and to regard the summation thus obtained as a univariate exponential sum. We can then use the simultaneous approximation methods of Baker \cite{Bak1986}.

Finally, we use the irrationality of $\mu$ to obtain nontrivial cancellation on Davenport--Heilbronn minor arcs $\fm$ (this is where $|\balp|$ is of `intermediate' size). We need the information that $d, d-1 \in \cS$. These facts follow from our geometric assumptions. Indeed, to see that $d-1 \in \cS$ one may compare \eqref{Taylor2} to the Taylor expansion
\[
f_k(\bx + \bmu) = f_k(\bx) + f_k(\bmu) + \sum_{i=1}^{d-1} \mu^i \sum_{|\bi|_1 = i} \bi!^{-1} \partial^\bi f_k (\bx)
\]
about $\bx$, which shows that 
\[
F_{k,d-1} = (1,\ldots,1) \cdot \nabla f_k \qquad (1 \le k \le R).
\]
We thus have good rational approximations to $\balp$ and $\mu \balp$, and their strength may be used to contradict the irrationality of $\mu$ unless we have a nontrivial estimate on $\fm$.

The proof of Theorem \ref{TheoremTwo} is almost the same, with the only substantial change being a suitable analogue of Lemma \ref{Birch43}. It transpires that such an analogue can be deduced without much work from Schmidt's seminal paper \cite{Sch1985}. Further details shall be provided in \S \ref{SchmidtApproach}.

We organise thus. In \S \ref{BirchType}, we use Freeman's kernel functions to relate $N(P)$ to exponential sums; see \cite[\S 2]{Fre2002}. Using Birch's argument, we then obtain good simultaneous rational approximations to the $\alp_k$ ($1 \le k \le R$) in the case that $g(\balp, \bome_\diam)$ is `large'; see \cite[Lemma 4.3]{Bir1962}. In \S \ref{BakerType}, we simultaneously approximate the $\om_\bj$ ($1 \le |\bj|_1 \le d$). In \S \ref{special}, we use $S(\balp)$ to obtain simultaneous rational approximations to the $\alp_{k,j}$ ($1 \le k \le R, j \in \cS$). In \S \ref{classical}, we adapt classical bounds to the present context. In \S \ref{bgf}, we exploit the irrationality of $\mu$ by using a simplification of the methods of Bentkus, G\"otze and Freeman, similarly to \cite[\S 2]{Woo2003}. The lemmas therein motivate our precise Davenport--Heilbronn trisection, which we present in \S \ref{dh}. We then resolve the asymptotic formula \eqref{asymp}. We complete the proof of Theorem \ref{MainThm} in \S \ref{SingularIntegral} by establishing the final statement of the theorem. It is then that we provide Schmidt's interpretation \cite{Sch1982, Sch1985} of the singular integral $c$ as a real density. Finally, we prove Theorem \ref{TheoremTwo} in \S \ref{SchmidtApproach}.

We adopt the convention that $\eps$ denotes an arbitrarily small positive number, so its value may differ between instances. For $x \in \bR$ and $r \in \bN$, we put $e(x) = e^{2 \pi i x}$ and $e_r(x) = e^{2 \pi i x / r}$. Bold face will be used for vectors, for instance we shall abbreviate $(x_1,\ldots,x_n)$ to $\bx$, and define $|\bx| = \max(|x_1|, \ldots, |x_n|)$. For a vector $\bx$ of length $n$, and for $\bj \in \bZ_{\ge 0}^n$, we define $\bx^{\bj} = x_1^{j_1} \cdots x_n^{j_n}$, $|\bj|_1 = j_1 + \ldots + j_n$ and $\bj! = j_1! \cdots j_n!$. If $M$ is a matrix then we write $|M|$ for the maximum of the absolute values of its entries. We will use the unnormalised sinc function, given by $\sinc(x) = \sin(x)/x$ for $x \in \bR \setminus \{0\}$ and $\sinc(0) = 1$. 

We regard $\btau, \mu$ and $\eta$ as constants. The word \emph{large} shall mean in terms of $\bf, \eps$ and constants, together with any explicitly stated dependence. Similarly, the implicit constants in Vinogradov and Landau notation may depend on $\bf, \eps$ and constants, and any other dependence will be made explicit. The pronumeral $P$ denotes a large positive real number. The word \emph{small} will mean in terms of $\bf$ and constants. We sometimes use such language informally, for the sake of motivation; we make this distinction using quotation marks. 

The author thanks Trevor Wooley very much for his enthusiastic supervision, and for suggesting such an agreeable research programme. Special thanks go to Adam Morgan for an elegant proof of Lemma \ref{Adam}. Finally, thanks to the anonymous referees for doing a thorough job and making several helpful suggestions.

\section{Approximations of Birch type}
\label{BirchType}

We deploy the kernel functions introduced by Freeman \cite[\S 2.1]{Fre2002}; see also \cite[\S 2]{PW2014}. We shall define $T: [1, \infty) \to [1, \infty)$ in due course. For now, it suffices to note that 
\begin{equation} \label{Tbound}
T(P) \le P, 
\end{equation}
and that $T(P) \to \infty$ as $P \to \infty$. Put
\begin{equation} \label{Ldef}
L(P) = \max(1,\log T(P)), \qquad \rho = \eta L(P)^{-1}
\end{equation}
and
\begin{equation} \label{Kdef}
K_{\pm}(\alp) = \frac {\sin(\pi \alp \rho) \sin(\pi \alp(2 \eta \pm \rho))} {\pi^2 \alp^2 \rho}. 
\end{equation}
From \cite[Lemma 1]{Fre2002} and its proof, we have
\begin{equation} \label{Kbounds}
K_\pm(\alp) \ll \min(1, L(P) |\alp|^{-2})
\end{equation}
and
\begin{equation} \label{Ubounds}
0 \le \int_\bR e(\alp t) K_{-}(\alp)\d\alp \le U_\eta(t) \le \int_\bR e(\alp t) K_{+}(\alp)\d\alp \le 1,
\end{equation}
where
\begin{equation*}
U_\eta(t) = 
\begin{cases}
1, &\text{if } |t| < \eta \\
0, &\text{if } |t| \ge \eta.
\end{cases}
\end{equation*}

For $\balp \in \bR^R$, write
\begin{equation} \label{Kprod}
\bK_\pm(\balp) = \prod_{k \le R} K_\pm(\alp_k).
\end{equation}
The inequalities \eqref{Ubounds} give
\[
R_{-}(P)  \le N(P) \le R_+(P), 
\]
where 
\[ 
R_\pm(P) = \int_{\bR^R} S(\balp) e(-\balp \cdot \btau) \bK_\pm(\balp) \d \balp. 
\]
In order to prove \eqref{asymp}, it therefore remains to show that
\begin{equation} \label{goal1}
R_\pm(P) = (2 \eta)^R c P^{n-Rd} + o(P^{n-Rd})
\end{equation}
as $P \to \infty$, where $c$ is given by \eqref{cdef}. 

In this section we employ some classical bounds of Davenport \cite{Dav1959, Dav1962, Dav1963, Dav2005} and Birch \cite{Bir1962}; see also \cite[ch. 8]{Bro2009}. These results apply directly to Weyl sums associated to $\balp \cdot \bf$, and are proved by Weyl differencing down to degree one. As such, they are unaffected by the presence of terms of degree lower than $d$. The idea that lower order terms are irrelevant when establishing Weyl-type bounds is well known; Birch himself notes this in \cite[\S 2]{Bir1962}, and it was also used to prove \cite[Lemma 1]{DL1964}. From \eqref{gdef}, we see that the polynomial associated to the Weyl sum $g(\balp, \bom_\diam)$ has $\balp \cdot \bf$ as its highest degree component. Exploiting this, we may deduce these classical bounds for $g(\balp, \bome_\diam)$.

\begin{lemma} \label{Birch43}
Let $0 < \tet \le 1$. Suppose 
\begin{equation} \label{gbig0}
|g(\balp, \bome_\diam)| > P^{n- (n-\sig) \tet/ 2^{d-1} + \eps}.
\end{equation}
Then there exist integers $q, a_1, \ldots, a_R$ such that
\begin{equation} \label{qa}
1 \le q \le P^{R(d-1)\tet}, \qquad \gcd(a_1, \ldots, a_R, q) = 1
\end{equation}
and
\begin{equation} \label{FirstApprox}
2|q \balp - \ba| \le P^{R(d-1)\tet - d}. 
\end{equation}
In particular, if $|S(\balp)| > P^{n- (n-\sig) \tet/ 2^{d-1} + \eps}$ then there exist $q \in \bN$ and $\ba \in \bZ^R$ satisfying \eqref{qa} and \eqref{FirstApprox}. We may replace $g(\balp, \bome_\diam)$ by
\[
\sum_{1 \le x_1, \ldots, x_n \le P}  e\Bigl( \balp \cdot \bf(\bx) + \sum_{1 \le |\bj|_1 \le d-1} \omega_\bj \bx^\bj \Bigr),
\]
and the same conclusions hold. 
\end{lemma}

\begin{proof}
For the first statement we may imitate Birch's proof of \cite[Lemma 4.3]{Bir1962}. We have removed the implied constant from \eqref{gbig0} by redefining $\eps$ and recalling that $P$ is large. Now \eqref{Sg} gives rise to our second claim. The third assertion follows in the same way as the first.
\end{proof}

Throughout, put
\begin{equation} \label{kapDef}
\kap = \frac{n-\sig}{R(d-1)2^{d-1}}.
\end{equation}
It follows from \eqref{numvars} that 
\begin{equation} \label{kapBound}
\kap > R+1.
\end{equation}
The argument of the corollary to \cite[Lemma 4.3]{Bir1962} now produces the following.

\begin{cor} \label{Birch43cor}
For $\balp \in \bR^R$ with $|\balp| < P^{-d/2}$, and for $\bomega_\diam$ as in \eqref{omegaDiam}, we have
\begin{equation} \label{first}
g(\balp, \bom_\diam) \ll P^{n+\eps} (P^d |\balp|)^{-\kap}.
\end{equation}
\end{cor}

Fix a small positive real number $\tet_0$. Let $\fN$ be the set of $\balp \in \bR^R$ satisfying \eqref{qa} and \eqref{FirstApprox} with $\tet = \tet_0$, for some integers $q, a_1,\ldots,a_R$. Given $\balp \in \fN$, such integers would be unique. Indeed, if we also had 
\[
1 \le t \le P^{R(d-1)\tet_0}, \qquad (b_1, \ldots, b_R,t) = 1
\]
and
\[
2|t \balp - \bb| \le P^{R(d-1)\tet_0 - d}
\]
for some integers $t,b_1, \ldots, b_R$, then 
\begin{align*}
|q^{-1}\ba - t^{-1}\bb| &\le |\balp - q^{-1}\ba| + |\balp - t^{-1}\bb|  \\
&< (1/q+1/t) P^{R(d-1)\tet_0 -d} < (qt)^{-1};
\end{align*}
this would imply that $t^{-1} \bb= q^{-1} \ba$, and hence that $t=q$ and $\bb = \ba$.

Let $\fU$ be an arbitrary unit hypercube in $R$ dimensions. Using Lemma \ref{Birch43}, the argument of \cite[Lemma 4.4]{Bir1962} shows that
\begin{equation} \label{BasicMinorBound}
\int_{(\bR^R \setminus \fN) \cap \fU} |S(\balp)| \d \balp \ll P^{n - Rd - \eps}.
\end{equation}
Put
\begin{equation} \label{NstarDef}
\fN^* = \fN^*_P = \{ \balp \in \fN: |S(\balp)| > P^{n - R(R+1) d \tet_0} \}.
\end{equation}
The measure of $\fN \cap \fU$ is $O(P^{R(R+1)(d-1)\tet_0 - Rd})$, so
\[
\int_{(\fN \setminus \fN^*) \cap \fU} |S(\balp)| \d \balp \ll P^{n- Rd - R(R+1) \tet_0}.
\]
Combining this with \eqref{BasicMinorBound} yields
\[
\int_{(\bR^R \setminus \fN^*) \cap \fU} |S(\balp)| \d \balp \ll P^{n-Rd-\eps}.
\]
Now \eqref{Tbound}, \eqref{Ldef}, \eqref{Kbounds} and \eqref{Kprod} give
\[
\int_{\bR^R \setminus \fN^*} |S(\balp) \bK_\pm(\balp)| \d \balp \ll L(P)^R P^{n-Rd-\eps} = o(P^{n-Rd}).
\]
In view of the discussion surrounding \eqref{goal1}, it remains to show that
\begin{equation} \label{goal2}
\int_{\fN^*} S(\balp) e(-\balp \cdot \btau) \bK_\pm(\balp) \d \balp = (2 \eta)^R c P^{n-Rd} + o(P^{n-Rd})
\end{equation}
as $P \to \infty$, with $c$ as in \eqref{cdef}.

\section{Approximations of Baker type}
\label{BakerType}

By \eqref{gdef}, \eqref{Fdef}, \eqref{top} and \eqref{omegatop}, we have
\begin{equation} \label{omg}
g(\balp, \bom_\diam) = \sum_{|\bx| \le P} e \Bigl( \sum_{1 \le |\bj|_1 \le d} \ome_\bj \bx^\bj \Bigr).
\end{equation}
In the case that $g(\balp, \bome_\diam)$ is `large', we shall use \cite[Theorem 5.1]{Bak1986} to obtain simultaneous rational approximations to the $\omega_\bj$. The idea is to fix $x_2, \ldots, x_n$, so as to consider $\sum \ome_\bj \bx^\bj$ as a polynomial in $x_1$. If we simply do this, we are only able to approximate certain linear combinations of the $\om_\bj$, and we do not acquire enough information. However, if we first  change variables, then we can approximate different linear combinations of the $\om_\bj$. The point is to use several carefully selected changes of variables. We never actually make these changes of variables; we merely incorporate them into our summations.

Suppose we were to put $\bx = \by + x_1 \bm$, regarding $m_1 = 1, m_2, \ldots, m_n \in \bN$ and $y_1 = 0$ as being fixed. For some $\by$, will be able to simultaneously approximate the coefficients of the $x_1^j$ in $\sum_\bj \ome_\bj (\by + x_1 \bm)^\bj$. By the binomial theorem, the coefficient of $x_1^j$ in $(\by+x_1 \bm)^\bi$ is
\begin{equation} \label{zdef}
z_{j,\bi} = z_{j,\bi}(\bm, \by) := \sum_{\bj \le \bi: |\bj|_1 = j} {\bi \choose \bj} \bm^\bj \by^{\bi - \bj},
\end{equation}
where 
\[
{\bi \choose \bj} = \prod_{v \le n} {i_v \choose j_v},
\]
and $\bj \le \bi$ means that $j_v \le i_v$ ($1 \le v \le n$). Hence, the coefficient of $x_1^j$ in $\sum_\bj \ome_\bj (\by+x_1 \bm)^\bj$ is 
\begin{equation} \label{coeff}
 \sum_{|\bj|_1 = j} \bm^\bj \omega_\bj + 
\sum_{j < |\bi|_1 \le d} z_{j,\bi} \omega_\bi.
\end{equation}

Since we wish to approximate the $\omega_\bj$, the first sum in \eqref{coeff} motivates the need for our next lemma. Recall \eqref{NjDef}.

\begin{lemma} \label{Adam}
There exist $\bm_1, \ldots, \bm_{N_d} \in \bN^n$ such that the first entry of $\bm_t$ is $1$ $(1 \le t \le N_d)$ and the square matrices
\begin{equation} \label{MjDef}
M_j = (\bm_t^\bj)_{1 \le t \le N_j, |\bj|_1 = j} \qquad (1 \le j \le d)
\end{equation}
are invertible over $\bQ$.
\end{lemma}

\begin{proof} Put
\[
\bm_t = (\nu_1^{t-1}, \ldots, \nu_n^{t-1}) \qquad (1 \le t \le N_d)
\]
with $\nu_1 = 1$ and $\nu_s = 2^{(d+1)^{s-2}}$ ($2 \le s \le n$).  Let $j \in \{1,2,\ldots,d\}$, and note that the order of the vectors $\bj$ does not affect whether or not the matrix is invertible. We have
\[
\bm_t^\bj = (\bnu^\bj)^{t-1} \qquad (1 \le t \le N_j, |\bj|_1 = j),
\]
so $M_j$ is a square Vandermonde matrix with parameters $\bnu^\bj$ ($|\bj|_1 = j$), and it remains to show that if $|\bi|_1 = |\bj|_1 = j$ and $\bi \ne \bj$ then $\bnu^\bi \ne \bnu^\bj$. We may assume that $\bi > \bj$ in reverse lexicographic order, so that there exists $r \in \{2,3,\ldots,n\}$ such that $i_r > j_r$ and $i_s = j_s$ ($r+1 \le s \le n$). Now
\[
\bnu^\bi / \bnu^\bj \ge \nu_r / \nu_{r-1}^j > 1.
\]
\end{proof}

Henceforth, we let $\bm_1, \ldots, \bm_{N_d}$ be fixed vectors as in Lemma \ref{Adam}. Baker's work \cite[Theorem 5.1]{Bak1986} shows that if a Weyl sum in one variable is `large' then its non-constant coefficients admit good simultaneous rational approximations. There is currently no close analogue in many variables. However, since we have already restricted attention to a thin set of major arcs, we obtain a satisfactory analogue by fixing all but one variable and then using Baker's result. For the time being, we work with the more general Weyl sum $g(\balp, \bom_\diam)$. Put
\begin{equation} \label{Ndef}
N = N_1 + \ldots +  N_d.
\end{equation}

\begin{lemma} \label{bak}
Let $H > 0$ be such that 
\begin{equation} \label{Hlimit}
H^{2^d N + 1} \le P,
\end{equation}
and assume \eqref{gbig}. Then there exist unique $r \in \bN$ and 
\[
\ba_\dagger = (a_{\bj})_{1 \le |\bj|_1 \le d} \in \bZ^N
\]
such that
\begin{equation} \label{bak1}
r \ll H^{Nd} P^\eps, \qquad \gcd(r,\ba_\dagger) = 1
\end{equation}
and
\begin{equation} \label{bak2}
r \omega_\bj - a_\bj \ll H^{Nd} P^{\eps - |\bj|_1} \qquad 
(\bj \in \bZ_{\ge 0}^n: 1 \le |\bj|_1 \le d),
\end{equation}
where $\gcd(r,\ba_\dagger)$ denotes the greatest common divisor of $r$ and the entries of $\ba_\dagger$. 
\end{lemma}

\begin{proof}
Let $t \in \{1,2,\ldots, N_d\}$, and set $y_1 = 0$. By \eqref{omg}, we have
\[
g(\balp, \bom_\diam) = \sum_{\substack{y_2, \ldots, y_n: \\ |\by| \le (|\bm_t|+1) P}} 
\sum_{x_1 \in I_t(\by)} e \Bigl( \sum_{1 \le |\bj|_1 \le d} \om_\bj (\by + x_1 \bm_t)^\bj \Bigr),
\]
where
\[
I_t(\by) = \{ x_1 \in \bZ : |\by + x_1 \bm_t| \le P \}
\]
is a discrete subinterval of $[-P,P] \cap \bZ$. More precisely, given $t$ and $\by$ as above, there exists a real subinterval $[a,b]$ of $[-P,P]$ such that $I_t(\by) = [a,b] \cap \bZ$. By \eqref{gbig} and the triangle inequality, there exists $\by_t \in \bZ^n$ such that $|\by_t| \ll P$ and
\[
\Bigl | \sum_{x_1 \in I_t(\by_t)} e \Bigl( \sum_{1 \le |\bj|_1 \le d} \om_\bj (\by_t + x_1 \bm_t)^\bj \Bigr)  \Bigr | \gg PH^{-1}.
\]
Now \cite[Theorem 5.1]{Bak1986} and the calculation \eqref{coeff} imply the existence of integers $q_t, v_{t,d}, \ldots, v_{t,1}$ such that
\[
0 < q_t \ll H^d P^\eps
\]
and
\begin{equation} \label{q0errors}
q_t \Bigl(
 \sum_{|\bj|_1 = j} \bm_t^\bj \omega_\bj + 
 \sum_{j < |\bi|_1 \le d} z_{j,\bi,t} \omega_\bi
\Bigr) - v_{t,j} \ll H^d P^{\eps-j} \quad (1 \le j \le d),
\end{equation}
where 
\[
z_{j,\bi,t} = z_{j,\bi}(\bm_t, \by_t).
\]

With \eqref{MjDef}, put
\begin{equation} \label{DeljDef}
\Del_j = \det(M_j) \qquad (1 \le j \le d).
\end{equation}
In order for this to be well defined, we need to fix an ordering of the $\bj$ ($|\bj|_1 = j$), and we can do this by writing $\{ \bj_{1,j}, \ldots, \bj_{N_j,j} \}$ for the set of $\bj \in \bZ_{\ge 0}^n$ such that $|\bj|_1 = j$. Explicitly, we now have
\begin{equation} \label{MjExplicit}
M_j = \begin{pmatrix}
\bm_1^{\bj_{1,j}} & \ldots & \bm_1^{\bj_{N_j,j}} \\
\vdots && \vdots \\
\bm_{N_j}^{\bj_{1,j}} & \ldots & \bm_{N_j}^{\bj_{N_j,j}} 
\end{pmatrix} 
\qquad (1 \le j \le d).
\end{equation}
Note that the matrices $\Del_j M_j^{-1}$ have integer entries. For $j=1,2,\ldots,d$, write
\[
\Omega_j = \begin{pmatrix}
\omega_{\bj_{1,j}} \\
\vdots \\
\omega_{\bj_{N_j,j}}
\end{pmatrix},
\qquad
V_j = \begin{pmatrix}
v_{1,j} \\
\vdots \\
v_{N_j,j}
\end{pmatrix},
\]
and also let $\fQ_j = \diag(q_1, \ldots, q_{N_j})$. Let
\[
Q_j = q_1 \cdots q_{N_j} \qquad (1 \le j \le d)
\]
and
\[
\xi_j = \prod_{i=j}^d \Del_i Q_i \qquad (1 \le j \le d).
\]
For $j=1,2,\ldots,d$, put
\[
\psi_{t,j} = \sum_{j < |\bi|_1 \le d} z_{j,\bi,t} \omega_\bi,
\qquad
\Psi_j = \begin{pmatrix}
\psi_{1,j} \\
\vdots \\
\psi_{N_j,j}
\end{pmatrix}.
\]

We proceed, by induction on $|\bi|_1$ from $d$ down to $1$, to show that there exist integers $v_\bi$ ($1 \le |\bi|_1 \le d$) such that
\begin{equation} \label{InductiveOutcome}
\xi_{|\bi|_1} \omega_\bi - v_\bi \ll (H^d P^\eps)^{N_{|\bi|_1} + \ldots + N_d} P^{-|\bi|_1}.
\end{equation}
From \eqref{q0errors}, we have
\[
|\fQ_d M_d \Omega_d - V_d| \ll H^d P^{\eps-d}.
\]
Left multiplication by the integer matrix
\[
\Del_d Q_d M_d^{-1} \fQ_d^{-1} = (\Del_d  M_d^{-1}) \cdot (Q_d \fQ_d^{-1})
\]
gives
\[
|\Del_d Q_d \Omega_d - \Del_d Q_d M_d^{-1} \fQ_d^{-1} V_d| \ll (H^d P^\eps)^{N_d} P^{-d},
\]
since $q_t \ll H^d P^\eps$ ($1 \le t \le N_d$). In particular, there exist $v_\bj \in \bZ$ ($|\bj|_1 = d$) such that 
\begin{equation*}
\Delta_d Q_d \omega_\bj - v_\bj \ll (H^d P^\eps)^{N_d} P^{-d} \qquad (|\bj|_1 = d).
\end{equation*}
We have confirmed \eqref{InductiveOutcome} whenever $|\bi|_1 = d$.

Next let $j \in \{1, 2, \ldots, d-1 \}$, and suppose that for $i \in \{j+1, j+2, \ldots, d\}$ there exist $v_\bi \in \bZ$ ($|\bi|_1 = i$) satisfying \eqref{InductiveOutcome}. Put
\[
Z_j = \begin{pmatrix}
z_{1,j} \\
\vdots \\
z_{N_j,j}
\end{pmatrix},
\]
where for $1 \le t \le N_j$ we write
\[
z_{t,j} = \sum_{i = j+1}^d \Del_{j+1} \cdots \Del_{i-1} Q_{j+1} \cdots Q_{i-1} \sum_{|\bi|_1 = i} z_{j,\bi,t} v_\bi.
\]
From \eqref{q0errors}, we see that
\[
|\fQ_j  (M_j \Omega_j + \Psi_j) - V_j| \ll H^d P^{\eps-j}.
\]
Noting that 
\[
|\Del_jQ_j M_j^{-1} \fQ_j^{-1}| = |(\Del_j M_j^{-1}) \cdot (Q_j \fQ_j^{-1})| \ll (H^d P^\eps)^{N_j-1},
\]
we now have
\[
|\Del_j Q_j \Omega_j - \Del_jQ_j M_j^{-1} \fQ_j^{-1} (V_j - \fQ_j \Psi_j)| \ll (H^d P^\eps)^{N_j} P^{-j}.
\]
Hence
\begin{align} \notag 
\xi_j \Omega_j
&=
\xi_{j+1} \Del_jQ_j M_j^{-1} \fQ_j^{-1} (V_j - \fQ_j \Psi_j) 
+ O((H^d P^\eps)^{N_j + \ldots + N_d} P^{-j}) \\
\label{MainCalc} &= X_j  - \Del_j Q_j M_j^{-1} \xi_{j+1} \Psi_j
+ O((H^d P^\eps)^{N_j + \ldots + N_d} P^{-j}),
\end{align}
where $X_j = \xi_{j+1} \Del_jQ_j M_j^{-1} \fQ_j^{-1} V_j$ has integer entries and we have used Landau's notation entry-wise. By our inductive hypothesis and the bound 
\[
z_{j,\bi,t} \ll P^{|\bi|_1-j},
\]
we have
\begin{equation} \label{NiceApprox}
\xi_{j+1} \Psi_j = 
Z_j 
+ O((H^d P^\eps)^{N_{j+1} + \ldots + N_d} P^{-j}).
\end{equation}
Substituting \eqref{NiceApprox} into \eqref{MainCalc} yields
\[ 
\xi_j \Omega_j = X_j - \Del_j Q_j M_j^{-1}Z_j
+ O((H^d P^\eps)^{N_j + \ldots + N_d} P^{-j}).
\]
In particular, there exist $v_\bj \in \bZ$ ($|\bj|_1 = j$) such that 
\[
\xi_j \omega_\bj - v_\bj \ll (H^d P^\eps)^{N_j + \ldots + N_d} P^{-j} \qquad (|\bj|_1 = j).
\]

The induction has shown that there exist integers $v_\bi$ ($1 \le |\bi|_1 \le d$) satisfying \eqref{InductiveOutcome}. Our existence statement follows by redefining $\eps$, and choosing $r, a_\bj$ ($1 \le |\bj|_1 \le d$) by rescaling the integers $\xi_1, (\xi_1 / \xi_{|\bj|_1}) v_\bj$ in such a way that $r > 0$ and $\gcd(r, \ba_\dagger) = 1$.

Next suppose \eqref{bak1} and \eqref{bak2} also hold with $s \in \bN$ and 
\[
\bb_\dagger =  (b_\bj)_ {1 \le |\bj|_1 \le d} \in \bZ^N
\]
in place of $r$ and $\ba_\dagger$. Then, by the triangle inequality, we have
\[
|a_\bj /r - b_\bj/s| \ll (1/r + 1/s) H^{Nd} P^{\eps - 1} \qquad (1 \le |\bj|_1 \le d).
\]
Since $P$ is large and $r,s \ll H^{Nd} P^\eps$, we may now recall \eqref{Hlimit} to see that
\[
|a_\bj/r - b_\bj/s| < (rs)^{-1} \qquad (1 \le |\bj|_1 \le d). 
\]
Hence $a_\bj/r = b_\bj/s$ ($1 \le |\bj|_1 \le d$). The conditions
\[
\gcd(r,\ba_\dagger) = \gcd(s, \bb_\dagger) = 1
\]
now imply that $(r,\ba_\dagger) = (s, \bb_\dagger)$. We have demonstrated uniqueness.
\end{proof}

It may be possible to obtain the inequalities \eqref{bak1} and \eqref{bak2} with a smaller power of $H$, but we do not require this. Using an argument similar to that of the corollary to \cite[Lemma 4.3]{Bir1962}, we now deduce the following estimate for $g(\balp, \bom_\diam)$.

\begin{cor} \label{BakCor0}
Let $\xi$ be a small positive real number. Let $\balp \in \bR^R$, and let 
\[
\bomega_\diam = (\omega_\bj)_{1 \le |\bj|_1 \le d-1} \in \bR^{N_1 + \ldots + N_{d-1}}
\]
be such that
\[
P^{|\bj|_1}|\omega_\bj| \le (P^{\xi + (2^dN + 1)^{-1}})^{Nd} \qquad (1 \le |\bj|_1 \le d-1).
\]
Then
\begin{equation} \label{second}
g(\balp, \bom_\diam) \ll_\xi P^{n+\xi} \Bigl( \max_{1 \le |\bj|_1 \le d-1} P^{|\bj|_1} |\omega_\bj| \Bigr)^{-(Nd)^{-1}}.
\end{equation}
\end{cor}

\begin{proof}
Let $\bj \in \bZ_{\ge 0}^n$ be such that $1 \le |\bj|_1 = j \le d-1$, and determine $H > 0$ by
\begin{equation} \label{Hdef}
P^j|\omega_\bj| = (HP^\xi)^{Nd}.
\end{equation}
Note that we have \eqref{Hlimit}. Assume for a contradiction that
\[
|g(\balp, \bom_\diam)| \ge P^{n+\xi} (P^j |\omega_\bj| )^{-(Nd)^{-1}},
\]
for some $P$ that is large in terms of $\xi$. Then $|g(\balp, \bom_\diam)| \ge  P^n H^{-1}$, so by Lemma \ref{bak} there exist $r,a_\bj \in \bZ$ satisfying $0< r \ll H^{Nd} P^\xi$ and
\begin{equation} \label{2sub}
r\omega_\bj - a_\bj \ll H^{Nd} P^{\xi-j}.
\end{equation}
The triangle inequality and \eqref{Hdef} now give
\[
a_\bj \ll H^{Nd} P^{\xi-j} + (H^{Nd} P^\xi) \cdot (H^{Nd} P^{Nd\xi-j}).
\]
By \eqref{Hlimit}, we must now have $a_\bj = 0$. Substituting this into \eqref{2sub} yields 
\[
\omega_\bj \ll H^{Nd} P^{\xi - j},
\]
contradicting \eqref{Hdef}. We must therefore have \eqref{second}.
\end{proof}

Put 
\begin{equation}\label{delDef}
\del = (R(R+1) Nd^2+1)\tet_0.
\end{equation}
Recall \eqref{partials} and \eqref{NstarDef}. We henceforth define the $\om_\bj$ ($1 \le |\bj|_1 \le d$) in terms of $\balp$ by \eqref{omegadef}. The following is another consequence of Lemma \ref{bak}.

\begin{cor} \label{BakCor}
Let $\balp \in \fN^*$. Then there exist unique
\[
r \in \bZ, \qquad \ba_\dagger = (a_{\bj})_{1 \le |\bj|_1 \le d} \in \bZ^N
\]
such that
\begin{equation} \label{bak1cor}
1 \le r < P^\del, \qquad \gcd(r,\ba_\dagger)=1
\end{equation}
and
\begin{equation} \label{bak2cor}
|r \omega_\bj -  a_\bj | < P^{\del - |\bj|_1} \qquad (1 \le |\bj|_1 \le d).
\end{equation}
There also exist unique integers $q, a_1, \ldots, a_R$ such that
\begin{equation} \label{qbound}
1 \le q \le P^{R(d-1)\tet_0}, \qquad \gcd(a_1, \ldots, a_R,q) = 1
\end{equation}
and
\begin{equation} \label{qerror}
2 |q\balp - \ba| \le P^{R(d-1)\tet_0 - d}.
\end{equation}
\end{cor}

\begin{proof} 
Recall that \eqref{Sg} holds with the specialisation \eqref{omegadef}. For existence of satisfactory $r$ and $\ba_\dagger$, apply Lemma \ref{bak} with $H = P^{R(R+1)d \tet_0}$. Our first uniqueness assertion follows in the same way as the uniqueness statement in Lemma \ref{bak}. Existence and uniqueness of $q,a_1,\ldots,a_R$ follow from the definition of $\fN$ and the subsequent discussion, since $\balp \in \fN^* \subseteq \fN$.
\end{proof}

\section{Special approximations}
\label{special}

Recall \eqref{NjDef}, and that the $\om_\bj$ are now defined in terms of $\balp$ by \eqref{omegadef}. Recall \eqref{Fdef}, and that $\cS$ is the set of $j \in \{1,2,\ldots,d\}$ such that $F_{1,j}, F_{2,j}, \ldots, F_{R,j}$ are linearly independent.

\begin{lemma} \label{SpecLemma}
Let $j \in \cS$ and $\balp \in \bR^R$. Let $r, a_\bj \in \bZ$ $(|\bj|_1 = j)$. Then there exist integers $D_j \ne 0$ and $a_{k,j}$ $(1 \le k \le R)$ such that
\[
D_j r \alp_k \mu^{d-j} - a_{k,j} \ll \max_{|\bj|_1 = j} |r \omega_\bj - a_\bj| \qquad (1 \le k \le R)
\]
and $D_j$ is bounded in terms of $\bf$.
\end{lemma}

\begin{proof}
As in the proof of Lemma \ref{bak}, we fix an ordering of the $\bj$ ($|\bj|_1 = j$) by writing $\{ \bj_{1,j}, \ldots, \bj_{N_j,j} \}$ for the set of $\bj \in \bZ_{\ge 0}^n$ such that $|\bj|_1 = j$. From \eqref{omegadef}, we have 
\[
\Omega_j = C_j Y_j,
\]
where
\[
\Omega_j = \begin{pmatrix}
\omega_{\bj_{1,j}} \\
\vdots \\
\omega_{\bj_{N_j,j}}
\end{pmatrix},
\quad
C_j = \begin{pmatrix}
d_{1,\bj_{1,j}} &\ldots& d_{R,\bj_{1,j}} \\
\vdots && \vdots \\
d_{1, \bj_{N_j,j}}&\ldots& d_{R, \bj_{N_j,j}}
\end{pmatrix},
\quad
Y_j =
\begin{pmatrix}
\alp_1 \mu^{d-j} \\
\vdots \\
\alp_R \mu^{d-j}
\end{pmatrix}.
\]
We note from \eqref{numvars} and \eqref{NjDef} that $N_j \ge R$. The condition $j \in \cS$ ensures that the $R$ columns of $C_j$ are linearly independent, and it follows from linear algebra that $C_j$ contains $R$ linearly independent rows, indexed say by $T_j \subseteq \{1,2,\ldots, N_j \}$ (row rank equals column rank). Form $C'_j$ by assembling these rows of $C_j$ to form an invertible $R \times R$ matrix, and let $A'_j =  (a_{\bj_{t,j}})_{t \in T_j}$ be the $R \times 1$ matrix formed by assembling the same rows of $(a_{\bj_{t,j}})_{1 \le t \le N_j}$. We put $D_j = \det(C'_j)$ and 
\[
\begin{pmatrix}
a_{1,j} \\
\vdots \\
a_{R,j}
\end{pmatrix}
= 
D_j (C'_j)^{-1} A'_j.
\]
Define the $R \times 1$ matrix $\Omega'_j = (\omega_{\bj_{t,j}})_{t \in T_j}$. Now $\Om'_j = C'_j Y_j$, so
\[
D_j r Y_j
-
\begin{pmatrix}
a_{1,j} \\
\vdots \\
a_{R,j}
\end{pmatrix}
=
D_j (C'_j)^{-1} (r \Omega'_j - A'_j),
\]
completing the proof.
\end{proof}

As $d,d-1 \in \cS$, we have the following corollary.

\begin{cor} \label{SpecCor}
Let $\balp \in \fN^*$. Let the integers $r$ and $a_\bj$ $(1 \le |\bj|_1 \le d)$ be as determined by Corollary \ref{BakCor}. Then there exists $C_\bf > 1$, depending only on $\bf$, as well as $D,E \in \bZ \setminus \{0\}$ and $\ba_1, \ba_2 \in \bZ^R$
such that 
\begin{equation} \label{DE}
|D|, |E| \le C_\bf, 
\end{equation}
\begin{equation} \label{Dr} 
|Dr \balp - \ba_1| \ll \max_{|\bj|_1 = d} |r \omega_\bj - a_\bj|
\end{equation}
and
\begin{equation} \label{Er}
|Er \mu \balp - \ba_2| \ll \max_{|\bj|_1 = d-1} |r \omega_\bj - a_\bj|.
\end{equation}
The choice of $(D, E, \ba_1, \ba_2)$ is unique if we impose the further conditions
\begin{equation} \label{further}
D,E > 0, \quad \gcd(D,\ba_1) = \gcd(E,\ba_2) = 1.
\end{equation}
\end{cor}

\begin{proof} For existence, apply Lemma \ref{SpecLemma} with $j=d$ and then with $j=d-1$. For uniqueness, suppose we also have \eqref{DE}, \eqref{Dr}, \eqref{Er} and \eqref{further} with $D',E',\ba_1',\ba_2'$ in place of $D,E,\ba_1,\ba_2$. Combining these bounds with \eqref{bak2cor} and the triangle inequality gives
\[
|D^{-1} \ba_1 - (D')^{-1} \ba'_1| < (DD')^{-1}
\]
and
\[
|E^{-1} \ba_2 - (E')^{-1} \ba'_2| < (EE')^{-1},
\]
so $D^{-1}\ba_1 = (D')^{-1}\ba'_1$ and $E^{-1}\ba_2 = (E')^{-1}\ba'_2$. Having made the assumptions \eqref{further} and
\[
D',E' > 0, \quad \gcd(D',\ba'_1) = \gcd(E',\ba'_2) = 1,
\]
we must now have $(D',E',\ba_1',\ba_2') = (D,E,\ba_1,\ba_2)$.
\end{proof}

Henceforth, fix $C_\bf$ to be as in Corollary \ref{SpecCor}. Recall \eqref{Ndef}. For $r,D,E,q \in \bN$,
\[
\ba_\dagger = (\ba_\bj)_{1 \le |\bj|_1 \le d} \in \bZ^N, \quad \ba_1, \ba_2,\ba \in \bZ^R,
\]
write
\begin{equation} \label{suppress}
\cX = (r,D,E,q,\ba_\dagger,\ba_1, \ba_2,\ba),
\end{equation}
and let
$\fR(\cX) = \fR_P(\cX) $ be the set of $\balp \in \bR^R$ satisfying \eqref{bak1cor}, \eqref{bak2cor}, \eqref{qbound}, \eqref{qerror}, \eqref{DE}, \eqref{Dr}, \eqref{Er} and \eqref{further}. Let $\fR = \fR_P$ be the union of these sets. This union is disjoint, as with uniqueness in Corollaries \ref{BakCor} and \ref{SpecCor}. These corollaries also tell us that
\begin{equation} \label{subset}
\fN^* \subseteq \fR.
\end{equation} 
Recall \eqref{partials}.

\begin{lemma} \label{identities}
Suppose $\fR_P(\cX) \ne \emptyset$. Then
\begin{equation} \label{equality1}
(Dr)^{-1} \ba_1 = q^{-1} \ba,
\end{equation}
and $q$ divides $Dr$. We must also have
\begin{equation} \label{equality2}
r^{-1} a_\bj = q^{-1} \sum_{k \le R} d_{k,\bj} a_k \qquad (|\bj|_1 = d).
\end{equation}
\end{lemma}

\begin{proof}
Let $\balp \in \fR_P(\cX)$. From \eqref{bak2cor}, \eqref{qerror}, \eqref{Dr} and the triangle inequality, we see that
\[
|(Dr)^{-1} \ba_1 - q^{-1} \ba| \ll q^{-1}  P^{R(d-1)\tet_0 - d} + r^{-1} P^{\del - d}.
\]
By \eqref{bak1cor}, \eqref{qbound} and \eqref{DE}, we now have
\[
|(Dr)^{-1} \ba_1 - q^{-1} \ba| < (Drq)^{-1},
\]
which implies \eqref{equality1}. Hence $q$ divides $Dr$, since $\gcd(a_1, \ldots, a_R, q) = 1$. Now let $\bj \in \bZ_{\ge 0}^n$ be such that $|\bj|_1 = d$. We see from \eqref{omegadef} and \eqref{bak2cor} that
\[
\Bigl |\sum_{k \le R} d_{k,\bj} \alp_k - a_\bj / r \Bigr | < r^{-1} P^{\del-d}.
\]
Combining this with \eqref{qerror} and the triangle inequality yields
\[
r^{-1} a_\bj - q^{-1} \sum_{k \le R} d_{k,\bj} a_k \ll r^{-1} P^{\del-d} + q^{-1} P^{R(d-1)\tet_0 -d}.
\]
In light of \eqref{bak1cor} and \eqref{qbound}, we now have
\[
\Bigl | r^{-1} a_\bj - q^{-1} \sum_{k \le R} d_{k,\bj} a_k \Bigr| < (qr)^{-1},
\]
which establishes \eqref{equality2}.
\end{proof}

\section{Adaptations of known bounds}
\label{classical}

In this section we consider $S(\balp)$ for $\balp \in \fR$. Let $r, D, q \in \bN$, where $D \le C_\bf$ and $q$ divides $Dr$. Assume that $Dr \le P$. Let $\ba \in \bZ^R$ and
\begin{equation} \label{dagger}
\ba_\dagger = (a_{\bj})_{1 \le |\bj|_1 \le d} \in \bZ^N,
\end{equation}
where we recall \eqref{NjDef} and \eqref{Ndef}. Recall that the $\om_\bj$ are defined in terms of $\balp$ by \eqref{omegadef}, and put 
\begin{equation} \label{spec1}
\balp = q^{-1} \ba + \bz, \qquad \omega_\bj = r^{-1} a_\bj  + z_\bj \quad (1 \le |\bj|_1 \le d-1).
\end{equation}
Recall \eqref{gdef} and \eqref{Sg}. Our starting point is the calculation
\begin{equation} \label{gnote}
g(\balp, \bome_\diam) 
= \sum_{\bx \mmod Dr} e \Bigl( q^{-1} \ba \cdot \bf(\bx)
+ r^{-1} \sum_{1 \le |\bj|_1 \le d-1} a_\bj \bx^\bj \Bigr) S_{Dr} (\bx),
\end{equation}
where
\[
S_{Dr}(\bx) = \sum_{\substack{|\by| \le P \\ \by \equiv \bx \mmod Dr}} 
e\Bigl( \bz \cdot \bf(\by) + \sum_{1 \le |\bj|_1 \le d-1} z_\bj \by^\bj \Bigr).
\]

For $\bgam \in \bR^R$ and 
\begin{equation} \label{diam}
\bgam_\diam = (\gam_\bj)_{1 \le |\bj|_1 \le d-1} \in \bR^{N_1 + \ldots + N_{d-1}},
\end{equation}
write
\begin{equation} \label{Idef}
I(\bgam, \bgam_\diam) = \int_{[-1,1]^n} e\Bigl( \bgam \cdot \bf(\bt) + \sum_{1 \le |\bj|_1 \le d-1} \gam_\bj \bt^\bj \Bigr) \d \bt.
\end{equation}
By \cite[Lemma 8.1]{Bro2009} and a change of variables, we have
\begin{equation} \label{SIcompare}
S_{Dr}(\bx) = (P/(Dr))^n I(\bgam, \bgam_\diam) + O((P/r)^{n-1} (1+|\bgam| + |\bgam_\diam|))
\end{equation}
with 
\begin{equation} \label{spec2}
\bgam = P^d \bz, \qquad
\gam_\bj = P^{|\bj|_1} z_\bj \qquad (1 \le |\bj|_1 \le d-1).
\end{equation}
Let
\begin{equation} \label{SraDef}
S_{r,D,q}(\ba, \ba_\dag) =\sum_{\bx \mmod Dr} e \Bigl( q^{-1} \ba \cdot \bf(\bx)
+ r^{-1} \sum_{1 \le |\bj|_1 \le d-1} a_\bj \bx^\bj \Bigr).
\end{equation}
Since $D \le C_\bf$, substituting \eqref{SIcompare} into \eqref{gnote} shows that
\begin{equation} \label{one}
g(\balp, \bom_\diam) - P^n (Dr)^{-n} S_{r,D,q}(\ba,\ba_\dagger)I(\bgam, \bgam_\diam) \ll rP^{n-1}(1+|\bgam| + |\bgam_\diam|),
\end{equation}
with \eqref{spec2}. Specialising $r=D=q=1$, $\ba = \bzero$, and $\ba_\dagger = \bzero$ yields
\begin{equation} \label{two}
g(\balp, \bom_\diam) = P^n I(\bgam, \bgam_\diam) + O(P^{n-1}(1+|\bgam| + |\bgam_\diam|))
\end{equation}
with 
\begin{equation} \label{WeirdGam}
\bgam = P^d \balp, \qquad \gam_\bj = P^{|\bj|_1} \omega_\bj \quad (1 \le |\bj|_1 \le d-1).
\end{equation}

Emulating \cite[Lemma 5.2]{Bir1962} or \cite[Lemma 8.8]{Bro2009}, we combine \eqref{first}, \eqref{second} and \eqref{two} in order to bound $I(\bgam, \bgam_\diam)$, uniformly for $\bgam \in \bR^R$ and $\bgam_\diam \in \bR^{N_1 + \ldots + N_{d-1}}$. 

\begin{lemma} \label{IboundLemma} Let $\lam$ be a small positive real number. Then for $\bgam \in \bR^R$ and $\bgam_\diam \in \bR^{N_1+N_2+ \ldots + N_{d-1}}$, we have
\begin{equation} \label{Ibound}
I(\bgam, \bgam_\diam) \ll_\lam (1 + |\bgam|^{\kap - \lam} + |\bgam_\diam|^{(Nd)^{-1} - \lam})^{-1}.
\end{equation}
\end{lemma}

\begin{proof} 
As $I(\bgam, \bgam_\diam) \ll 1$, we may assume that $|\bgam| + |\bgam_\diam|$ is large. Recall that \eqref{two} holds with \eqref{WeirdGam}. This, \eqref{first} and \eqref{second} show that
\[
I(\bgam, \bgam_\diam) \ll \frac{P^{\lam n^{-2} N^{-1}}} {|\bgam|^\kap + |\bgam_\diam|^{(Nd)^{-1}}} 
+ \frac{1+|\bgam| + |\bgam_\diam|}P
\]
whenever $|\bgam| < P^{d/2}$ and $|\bgam_\diam| \le (P^{\lam n^{-2} N^{-1} + (2^dN + 1)^{-1}})^{Nd}$. Recall \eqref{numvars}. As $|\bgam| + |\bgam_\diam|$ is large and $I(\bgam, \bgam_\diam)$ does not depend on $P$, we are free to choose $P = (|\bgam| + |\bgam_\diam|)^n$, which gives
\[
I(\bgam, \bgam_\diam) \ll \frac{(|\bgam| + |\bgam_\diam|)^{\lam (Nn)^{-1}}}
 {|\bgam|^\kap + |\bgam_\diam|^{(Nd)^{-1}}}.
\]
Recall \eqref{kapDef}. By cross-multiplying and considering cases, we may now deduce that
\[
I(\bgam, \bgam_\diam) \ll (|\bgam|^{\kap - \lam} + |\bgam_\diam|^{(Nd)^{-1}-\lam})^{-1}.
\]
As $|\bgam| + |\bgam_\diam| > 1$, this yields \eqref{Ibound}.
\end{proof}

In analogy with \cite[Lemma 15.3]{Dav2005}, we deduce the following bound. We note from Lemma \ref{identities} that the conditions below are necessarily met whenever $\fR_P(\cX) \ne \emptyset$.

\begin{lemma} \label{Sra} Let $\psi > 0$, $q \in \bN$ and $\ba \in \bZ^R$ be such that $\gcd(a_1, \ldots, a_R, q) = 1$. Let $D \in \bN$ with $D \le C_\bf$. Let $r \in \bN$ be such that $q$ divides $Dr$, and let $\ba_\dag$ be as in \eqref{dagger}. Then
\begin{equation} \label{SraBound}
S_{r,D,q}(\ba, \ba_\dag) \ll_\psi r^n q^{\psi - \kap}.
\end{equation}
\end{lemma}

\begin{proof}
We may assume without loss that $\psi < 1$. Since $|S_{r,D,q}(\ba, \ba_\dag)| \le (Dr)^n$, we may assume $q$ to be large in terms of $\psi$. Suppose for a contradiction that
\[
|S_{r,D,q}(\ba, \ba_\dag)| > (Dr)^n q^{\psi - \kap}.
\]
Break $S_{r,D,q}(\ba, \ba_\dag)$ into $(Dr/q)^n$ sums, parametrised by $\bv \in \{1,2,\ldots,Dr/q\}^n$. The sum associated to a given $\bv$ is
\begin{equation} \label{shape}
\sum_{1 \le y_1, \ldots, y_n \le q} e \Bigl( q^{-1} \ba \cdot \bf(\by + q \bv) +
\sum_{1 \le |\bj|_1 \le d-1} r^{-1} a_\bj (\by + q \bv)^\bj
\Bigr)
\end{equation}
and, by the triangle inequality, at least one such sum must exceed $q^{n + \psi - \kap}$ in absolute value. Fix $\bv \in \{1,2,\ldots,Dr/q\}^n$ so that the expression \eqref{shape} exceeds $q^{n+\psi - \kap}$ in absolute value.

The polynomial in the Weyl sum \eqref{shape} is of the shape $q^{-1} \ba \cdot \bf(\by)$ plus lower degree terms. By \eqref{kapDef}, we may apply Lemma \ref{Birch43} with $P=q$ and $\tet = R^{-1}(d-1)^{-1} - \psi/n$. This shows that there exist integers $s, b_1, \ldots, b_R$ such that
\[
1 \le s < q, \qquad |sa_k / q - b_k| < q^{-1} \qquad (1 \le k \le R).
\]
Hence $a_k/q = b_k/s$ ($1 \le k \le R$). This is impossible, since $0 < s < q$ and $\gcd(a_1, \ldots, a_R,q) = 1$. This contradiction implies \eqref{SraBound}.
\end{proof}

In view of \eqref{subset}, we may restrict attention to $\fR$. With \eqref{one} as the harbinger of our endgame, we perceive the need to obtain a nontrivial upper bound for $S_{r,D,q}(\ba, \ba_\dag) \cdot I(\bgam, \bgam_\dag)$ on Davenport--Heilbronn minor arcs. From \eqref{kapBound}, \eqref{spec1} and \eqref{spec2}, we see that the inequalities \eqref{Ibound} and \eqref{SraBound} save a `large' power of $P^d|q\balp - \ba|$ on $\fR$. We shall also need to save a power of $P^{d-1}|Er \mu \balp - \ba_2|$. If $|\balp|$ is somewhat large, the irrationality of $\mu$ will force one of $|q\balp - \ba|$ and $|Er \mu \balp - \ba_2|$ to be somewhat large, leading to a nontrivial estimate.

From \eqref{Er}, \eqref{spec1}, \eqref{diam} and \eqref{spec2}, we see that \eqref{Ibound} saves a power of 
$P^{d-1} |E\mu \balp - r^{-1} \ba_2|$ over a trivial estimate for $I(\bgam, \bgam_\diam)$. Thus, our final task for this stage of the analysis is to save a power of $r$ over a trivial estimate for $S_{r,D,q}(\ba, \ba_\dag)$. Roughly speaking, we achieve this by fixing $x_2, \ldots, x_n$ and then using \cite[Theorem 7.1]{Vau1997} to bound the resulting univariate exponential sum. This entails bounding the greatest common divisor of the coefficients of this latter sum, which leads us to consider several notional changes of variables, much like in the proof of Lemma \ref{bak}. 

\begin{lemma}
Let $D \in \bN$ with $D \le C_\bf$. Let $r,q \in \bN$, and let $\psi > 0$. Let $\ba \in \bZ^R$, and let $\ba_\dag$ be as in \eqref{dagger}. Assume \eqref{equality2}, and that $\gcd(r,\ba_\dag) = 1$. Then
\begin{equation} \label{SraBound2}
S_{r,D,q}(\ba, \ba_\dag) \ll_\psi r^{n-(N_d d)^{-1}+\psi}.
\end{equation}
\end{lemma}

\begin{proof}
By \eqref{fknote}, \eqref{equality2}, \eqref{SraDef} and periodicity, we have
\begin{equation} \label{periodicity}
S_{r,D,q}(\ba, \ba_\dag)  = D^n \sum_{\bx \mmod r} e_r \Bigl( \sum_{1 \le |\bj|_1 \le d} a_\bj \bx^\bj \Bigr).
\end{equation}
Set $y_1 = 0$, and recall that we have fixed $\bm_1, \ldots, \bm_{N_d} \in \bN^n$ as in Lemma \ref{Adam}. Write 
\[
\bm_t = (m_{t,1}, \ldots, m_{t,n}) \qquad (1 \le t \le N_d),
\]
where $m_{t,1} = 1$ ($1 \le t \le N_d$). Equation \eqref{periodicity} and the triangle inequality give
\begin{equation} \label{triangle}
|S_{r,D,q}(\ba, \ba_\dag)| \le D^n \sum_{\substack{y_2, \ldots, y_n: \\ |\by| \le (|\bm_t|+1) r}}  |S(\bm_t, \by)|
\qquad (1 \le t \le N_d),
\end{equation}
where
\[
S(\bm_t, \by) = \sum_{x_1 \in I_r(\bm_t, \by)} e_r\Bigl(
\sum_{1\le |\bj|_1 \le d} \ba_\bj (\by + x_1 \bm_t)^\bj \Bigr);
\]
here
\[
I_r(\bm_t, \by) = \{ x_1 \in \bZ : 1 \le y_1 + m_{t,1} x_1, \ldots, y_n+m_{t,n} x_1 \le r \}
\]
is a discrete subinterval of $\{1,2,\ldots,r\}$. More precisely, given $t$ and $\by$ as above, there exists a real subinterval $[a,b]$ of $[1,r]$ such that $I_r(\bm_t, \by) = [a,b] \cap \bZ$. Suppose for a contradiction that 
\[
|S_{r,D,q}(\ba, \ba_\dag)| > D^n r^{n- (N_d d)^{-1} + \psi} \prod_{t \le N_d} (2|\bm_t| +3)^{n-1},
\]
and that $r$ is large in terms of $\psi$. Then, by \eqref{triangle}, there exist $\by_1, \ldots, \by_{N_d} \in \bZ^n $ such that
\begin{equation} \label{Sbig}
|S(\bm_t, \by_t)| > r^{1 - (N_d d)^{-1} + \psi} \qquad (1 \le t \le N_d).
\end{equation}

In view of the calculation \eqref{zdef}, we see that if $1 \le t \le N_d$ and $1 \le j \le d$ then the coefficient of $x_1^j$ in
\[
\sum_{1 \le |\bj|_1 \le d} a_\bj (\by_t + x_1 \bm_t)^\bj
\]
is
\begin{equation} \label{ctj}
c_{t,j} := \sum_{|\bj|_1 = j} a_\bj \bm_t^\bj + \sum_{j < |\bi|_1 \le d} a_\bi z_{j,\bi} (\bm_t, \by_t).
\end{equation}
At this point we apply \cite[Theorem 7.1]{Vau1997}. It is necessary to remove any common divisors of $r,c_{t,1}, \ldots, c_{t,d}$. Moreover, since \cite[Theorem 7.1]{Vau1997} deals with complete exponential sums, we use an estimate due to Hua \cite[\S 3]{Hua1940} to compare our incomplete exponential sum to the corresponding complete exponential sum. Thus, it follows that
\[
S(\bm_t, \by_t) \ll \gcd(r,c_{t,1}, \ldots, c_{t,d})^{1/d-\eps} r^{1-1/d+\eps} \qquad (1 \le t \le N_d).
\]

Coupling this with \eqref{Sbig}, we deduce that
\[
\gcd(r,c_{t,1}, \ldots, c_{t,d}) \gg r^{1- 1/N_d + \psi} \qquad (1 \le t \le N_d),
\]
so
\begin{equation} \label{gcdBig0}
\prod_{t \le N_d} \gcd(r,c_{t,1}, \ldots, c_{t,d}) > r^{N_d - 1 + \psi}.
\end{equation}
By induction using the inequality 
\[
(a,b)(a,c) \le a \cdot \gcd(a,b,c) \qquad (a, b,c \in \bZ, a > 0),
\]
one can show that
\[
\prod_{t \le N_d} \gcd(r,c_{t,1}, \ldots, c_{t,d}) \le r^{N_d-1} G,
\]
where $G$ is the greatest common divisor of $r$ and the $c_{t,j}$ \mbox{($1 \le t \le N_d$, $1 \le j \le d$).} This and \eqref{gcdBig0} give
\begin{equation} \label{Gbig}
G > r^\psi.
\end{equation}

Note that 
\begin{equation} \label{decompose}
G \le \gcd(r, g_1, \ldots, g_d),
\end{equation}
where
\[
g_j = \gcd(c_{1,j}, \ldots, c_{N_j, j}) \qquad (1 \le j \le d).
\]
We adopt the notation of \eqref{DeljDef}, \eqref{MjExplicit} and the discussion in between. Write
\[
C_j = \begin{pmatrix}
c_{1,j} \\
\vdots \\
c_{N_j,j}
\end{pmatrix},
\qquad
A_j = \begin{pmatrix}
a_{\bj_{1,j}} \\
\vdots \\
a_{\bj_{N_j,j}}
\end{pmatrix}
\qquad (1 \le j \le d).
\]
We shall show by induction from $d$ down to $1$ that if $1 \le j \le d$ then
\begin{equation} \label{gcdInduct}
\gcd(g_j, \ldots, g_d) | \Del_j \cdots \Del_d G_j,
\end{equation}
where $G_j$ is the greatest common divisor of the $a_\bj$ ($j \le |\bj|_1 \le d$). Let $\cD_0$ be a common divisor of $c_{1,d}, \ldots, c_{N_d, d}$. From \eqref{ctj} we have
\[
C_d = M_d A_d.
\]
Hence
\[
\Del_d A_d = \Del_d M_d^{-1}  C_d,
\]
and so $\cD_0$ divides $\Del_d a_\bj$ ($|\bj|_1 = d$). We thus conclude that $g_d | \Del_d G_d$, thereby establishing the case $j=d$ of \eqref{gcdInduct}.

Now let $j \in \{1, 2, \ldots, d-1 \}$, and assume that 
\[
\gcd(g_i, \ldots, g_d) | \Del_i \cdots \Del_d G_i \qquad (j+1 \le  i \le d). 
\]
Let $\cD$ be a common divisor of $g_j, \ldots, g_d$. Then $\cD$ divides $c_{t,j}$ ($1 \le t \le N_j$), and our inductive hypothesis shows that 
\begin{equation} \label{IndImplies}
\cD | \Del_{j+1} \cdots \Del_d a_\bi \qquad (j < |\bi|_1 \le d).
\end{equation}
Equations \eqref{ctj} and \eqref{IndImplies} yield
\[
\Del_{j+1} \cdots \Del_d C_j \equiv \Del_{j+1} \cdots \Del_d  M_j A_j \mmod \cD,
\]
so
\[
\Del_j \cdots \Del_d A_j \equiv (\Del_j M_j^{-1}) \Del_{j+1} \cdots \Del_d C_j \equiv 
\begin{pmatrix}
0 \\
\vdots \\
0
\end{pmatrix} \mmod \cD.
\]
Coupling this with \eqref{IndImplies} yields $\cD | \Del_j \cdots \Del_d G_j$. Hence $\gcd(g_j, \ldots, g_d)$ divides $\Del_j \cdots \Del_d G_j$, and our induction is complete. We now have \eqref{gcdInduct}, in particular for $j=1$. Substituting this into \eqref{decompose} gives
\[
G \ll \gcd(r, G_1) = \gcd(r,\ba_\dag) = 1.
\]
This contradicts \eqref{Gbig}, thereby completing the proof the lemma.
\end{proof}

With \eqref{spec1}, we now specialise \eqref{spec2}. Define $S^*: \fR \to \bC$ as follows: if $\balp \in \fR(\cX)$ then
\begin{equation} \label{SstarDef}
S^*(\balp) = P^n (Dr)^{-n} S_{r, D,q}(\ba, \ba_\dag) I(\bgam, \bgam_\diam) e(\balp \cdot \bf(\bmu)). 
\end{equation}
We note from \eqref{delDef}, \eqref{bak1cor}, \eqref{bak2cor} and \eqref{qerror} that 
\[
r < P^\del, \qquad r|\bgam| < P^{2\del}, \qquad r|\bgam_\diam| < P^\del.
\]
By \eqref{Sg} and \eqref{one}, we now have
\begin{equation} \label{SstarCompare}
S(\balp) = S^*(\balp) + O(P^{n-1+2\del}) \qquad (\balp \in \fR).
\end{equation}
Let $\fU$ be an arbitrary unit hypercube in $R$ dimensions. The measure of $\fN^* \cap \fU$ is $O(P^{R(R+1)(d-1)\tet_0 - Rd})$, so \eqref{delDef}, \eqref{subset} and \eqref{SstarCompare} show that
\[
\int_{\fN^* \cap \fU} |S(\balp) - S^*(\balp)| \d \balp \ll P^{n - Rd - 1 + 3 \del}.
\]
Since $\del$ is small, we now see from \eqref{Tbound}, \eqref{Ldef}, \eqref{Kbounds} and \eqref{Kprod} that
\begin{align} \notag 
\int_{\fN^*} S(\balp) e(-\balp \cdot \btau) \bK_{\pm}(\balp) \d \balp &= \int_{\fN^*} S^*(\balp) e(-\balp \cdot \btau) \bK_{\pm}(\balp) \d \balp \\
\label{IntCompare} &\qquad+ o(P^{n - Rd}).
\end{align}

Let $\balp \in \fR(\cX)$. Equations \eqref{spec1} and \eqref{spec2} give $q \bgam = P^d (q\balp - \ba)$ and
\[
r \gam_\bj = P^{d-1} (r \omega_\bj - a_\bj) \qquad (|\bj|_1 = d-1).
\]
Thus, by \eqref{Er}, \eqref{Ibound}, \eqref{SraBound}, \eqref{SraBound2} and \eqref{SstarDef}, we have
\[
S^*(\balp) \ll P^n(q+P^d |q \balp - \ba|)^{\eps - \kap}
\]
and
\[
S^*(\balp) \ll P^n(r+P^{d-1} |Er \mu \balp - \ba_2 |)^{\eps - (Nd)^{-1}}.
\]
In light of \eqref{kapBound} and the bound $E \le C_\bf$, we now have
\begin{equation} \label{SstarBound}
S^*(\balp) \ll P^n (q+P^d |q \balp - \ba|)^{- R - 1 - \eps} F(\balp)^\eps,
\end{equation}
where
\begin{align} \notag
F(\balp) &= F(\balp; P) \\
 \label{fdef} &= (q+P^d |q \balp - \ba|)^{-1} (Er+P^{d-1} |Er \mu \balp - \ba_2|)^{-1}
\end{align}
is well defined on $\fR = \fR_P$.

\section{Lemmas of Freeman type}
\label{bgf}

The saving of $(q+P^d|q \balp - \ba|)^{R+1+\eps}$ in \eqref{SstarBound} suffices to obtain an upper bound for 
\[
\int_{\fN^*} S^*(\balp) e(-\balp \cdot \btau) \bK_{\pm}(\balp) \d \balp
\]
of the correct order of magnitude. On Davenport--Heilbronn minor arcs, however, we shall need to save slightly more. Using the methods of Bentkus, G\"otze and Freeman, as exposited in \cite[Lemmas 2.2 and 2.3]{Woo2003}, we will show that $F(\balp) =  o(1)$ in the case that $|\balp|$ is of `intermediate' size, where $F(\balp)$ is as in \eqref{fdef}. The set on which we are able to prove this estimate will define our Davenport--Heilbronn minor arcs. The success of our endeavour depends crucially on the irrationality of $\mu$.

For the argument to work, we need to essentially replace $F$ by a function defined on all of $\bR^R$. For $\balp \in \bR^R$, let $\cF(\balp; P)$ be the supremum of the quantity
\[
(q+P^d |q \balp - \ba|)^{-1} (s + P^{d-1} |s \mu \balp - \bb|)^{-1}
\]
over $q,s \in \bN$ and $\ba, \bb \in \bZ^R$ satisfying $q \le C_\bf s$. It follows from Lemma \ref{identities} and the bound $D \le C_\bf$ that
\begin{equation} \label{Fbound}
F(\balp; P) \le \cF(\balp; P) \qquad (\balp \in \fR_P).
\end{equation}

\begin{lemma} \label{Freeman1}
Let $V$ and $W$ be fixed real numbers such that $0 < V \le W$. Then
\begin{equation} \label{Freeman1eq}
\sup \{ \cF(\balp; P): V \le |\balp| \le W \} \to 0 \qquad (P \to \infty).
\end{equation}
\end{lemma}

\begin{proof} Suppose for a contradiction that \eqref{Freeman1eq} is false. Then there exist $\psi > 0$ and 
\[
(\balp^{(m)}, P_m, q_m, s_m, \ba^{(m)}, \bb^{(m)}) \in \bR^R \times [1, \infty) \times \bN^2 \times (\bZ^R)^2
\quad (m \in \bN) 
\]
such that (i) the sequence $(P_m)$ increases monotonically to infinity, (ii)
\[
V \le |\balp^{(m)}| \le W \qquad (m \in \bN)
\]
and (iii) if $m \in \bN$ then
\begin{equation} \label{Fre1key}
(q_m+P_m^d|q_m \balp^{(m)} - \ba^{(m)}|) \cdot (s_m +P_m^{d-1} |s_m \mu \balp^{(m)} - \bb^{(m)}|) < \psi^{-1}.
\end{equation}
Now $q_m, s_m < \psi^{-1} \ll 1$, so $|\ba^{(m)}|, |\bb^{(m)}| \ll 1$. In particular, there are only finitely many possible choices for the tuple $(q_m, s_m, \ba^{(m)}, \bb^{(m)})$, so this tuple must take a particular value infinitely often, say $(q,s,\ba,\bb)$. Note that $\ba \ne \bzero$, for if $m$ is large then \eqref{Fre1key} and the condition $|\balp^{(m)}| \ge V$ ensure that $\ba^{(m)} \ne \bzero$. 

Let $k \in \{1,2,\ldots,R\}$ be such that $a_k \ne 0$. From \eqref{Fre1key} we have
\[
\alp^{(m)}_k - q_m^{-1} a^{(m)}_k  \ll P_m^{-d}, \qquad \mu \alp^{(m)}_k - s_m^{-1} b^{(m)}_k  \ll P_m^{1-d}.
\]
Hence
\[
\mu q_m^{-1} a^{(m)}_k -  s_m^{-1} b^{(m)}_k \ll P_m^{1-d} \to 0 \qquad (m \to\infty).
\]
We conclude that
\[
\mu = \frac{qb_k}{sa_k},
\]
contradicting the irrationality of $\mu$. This contradiction establishes \eqref{Freeman1eq}.
\end{proof}

\begin{cor} \label{Freeman2}
There exists $T: [1,\infty) \to [1,\infty)$, increasing monotonically to infinity, such that
\begin{equation} \label{Tnote}
T(P) \le P^\del
\end{equation}
and, for large $P$,
\begin{equation} \label{FreemanBound}
\sup \{ F(\balp; P) : \balp \in \fN^*_P, \: P^{\del-d} \le |\balp| \le T(P) \} \le T(P)^{-1}.
\end{equation}
\end{cor}

\begin{proof} 
Recall \eqref{subset} and \eqref{Fbound}. We shall prove, \emph{a fortiori}, that
\[
\sup \{ \cF(\balp; P) : P^{\del - d} \le |\balp| \le T(P) \} \le T(P)^{-1}.
\]
Lemma \ref{Freeman1} yields a sequence $(P_m)$ of positive real numbers such that if
\[
1/m \le |\balp| \le m
\]
then $\cF(\balp; P_m) \le 1/m$. We may choose this sequence to be increasing, and such that if $m \in \bN$ then $P_m^\del \ge m$. We define $T$ by $T(P) = 1$  ($1 \le P < P_1$) and $T(P) = m$ ($P_m \le P < P_{m+1}$). We note \eqref{Tnote}, and that $T$ increases monotonically to infinity. Now
\[
\sup \{ \cF(\balp; P): T(P)^{-1} \le |\balp| \le T(P) \} \le T(P)^{-1},
\]
for if $P \ge P_m$ then $\cF(\balp; P) \le \cF(\balp; P_m)$.

It remains to show that if $P$ is large and
\begin{equation} \label{Tass}
|\balp| < T(P)^{-1} < \cF(\balp; P)
\end{equation}
then $|\balp| < P^{\del - d}$. Suppose $P$ is large and $\balp \in \bR^R$ satisfies \eqref{Tass}. Then
\begin{equation} \label{Tass2}
(q+P^d |q \balp - \ba|) \cdot (s+P^{d-1} |s \mu \balp  - \bb|) < T(P)
\end{equation}
for some $q,s \in \bN$ and some $\ba, \bb \in \bZ^R$ satisfying $q \le C_\bf s$. We must therefore have $q + P^d|q\balp - \ba| < T(P)^{1/2}$ or $s + P^{d-1}|s \mu \balp - \bb| < T(P)^{1/2}$.

\textbf{Case: $q + P^d|q\balp - \ba| < T(P)^{1/2}$.}
Now $q < T(P)^{1/2}$ and
\[
|q\balp - \ba| < P^{-d} T(P)^{1/2}.
\]
Combining these with \eqref{Tnote}, \eqref{Tass} and the triangle inequality yields 
\[
|\ba| < T(P)^{-1/2} + P^{-d} T(P)^{1/2} \to 0 \qquad (P \to \infty).
\]
Hence $\ba = \bzero$, so
\[
|\balp| < P^{-d} T(P)^{1/2} \le P^{\del-d},
\]
as desired.

\textbf{Case: $s + P^{d-1}|s \mu \balp - \bb| < T(P)^{1/2}$.} In this case $s < T(P)^{1/2}$ and 
\[
|s \mu \balp - \bb| < P^{1-d}T(P)^{1/2}.
\]
By \eqref{Tnote}, \eqref{Tass} and the triangle inequality, we now have
\[ 
|\bb| \ll T(P)^{-1/2} + P^{1-d}T(P)^{1/2} \to 0 \qquad (P \to \infty),
\]
so $\bb = \bzero$. Thus
\[
|q\balp| \le C_\bf s |\balp| \ll |s\mu \balp| \ll P^{1-d}T(P)^{1/2}.
\]
Combining this with \eqref{Tnote}, \eqref{Tass2} and the triangle inequality yields
\[
|\ba| \ll P^{1-d}T(P)^{1/2} + P^{-d}T(P) \to 0 \qquad (P \to \infty),
\]
so $\ba= \bzero$. Substituting this into \eqref{Tass2} and using \eqref{Tnote} gives 
\[
|\balp| \le |q\balp| < P^{-d}T(P) \le P^{\del -d},
\]
completing the proof.
\end{proof}

\section{The Davenport--Heilbronn method}
\label{dh}

In this section we finish the proof of the asymptotic formula \eqref{asymp}. Recall that it remains to prove \eqref{goal2}. By \eqref{IntCompare}, it now suffices to show that
\begin{equation} \label{goal3}
\int_{\fN^*} S^*(\balp) e(-\balp \cdot \btau) \bK_{\pm}(\balp) \d \balp = (2\eta)^R c P^{n-Rd} + o(P^{n-Rd})
\end{equation}
as $P \to \infty$, where $c$ is given by \eqref{cdef}. With $T(P)$ as in Corollary \ref{Freeman2}, we define our Davenport--Heilbronn major arc by
\[
\fM_1 = \{ \balp \in \bR^R: |\balp| < P^{\del-d} \},
\]
our minor arcs by
\[
\fm = \{ \balp \in \bR^R:  P^{\del-d} \le |\balp| \le T(P) \}
\]
and our trivial arcs by
\[
\ft = \{ \balp \in \bR^R:  |\balp| > T(P) \}.
\]

Recall that to any $\balp \in \fN^*$ we have uniquely assigned $q \in \bN$ and $\ba \in \bZ^R$ via \eqref{qbound} and \eqref{qerror}. For any unit hypercube $\fU$ in $R$ dimensions, we have
\[
\int_{\fN^* \cap \fU} P^n (q+P^d|q\balp - \ba|)^{-R-1-\eps} \d \balp \ll P^n X_1 Y_1,
\]
where
\[
X_1 = \sum_{q \in \bN} q^{-1-\eps} \ll 1
\]
and
\[
Y_1 = \int_{\bR^R} (1+P^d|\bbet|)^{- R - 1} \d \bbet 
\le \Bigl( \int_\bR (1+P^d |\bet|)^{-1-1/R} \d \bet \Bigr)^R \ll P^{-Rd}.
\]
Hence 
\begin{equation} \label{crux} 
\int_{\fN^* \cap \fU} P^n (q+P^d|q\balp - \ba|)^{-R-1-\eps} \d \balp \ll P^{n-Rd}.
\end{equation}
Combining this with \eqref{SstarBound} and \eqref{FreemanBound} gives
\[
\int_{\fN^* \cap \fm \cap \fU} |S^*(\balp)| \d \balp
\ll \sup_{\balp \in \fN^* \cap \fm} F(\balp)^\eps \cdot P^{n-Rd}  
\ll T(P)^{-\eps} P^{n-Rd}.
\]
In view of \eqref{Ldef}, \eqref{Kbounds} and \eqref{Kprod}, we now have
\begin{equation} \label{minor}
\int_{\fN^* \cap \fm} |S^*(\balp)  \bK_\pm(\balp)| \d \balp \ll L(P)^R T(P)^{-\eps} P^{n-Rd} = o(P^{n-Rd}).
\end{equation}
Note that 
\begin{equation} \label{Fnote}
0 < F(\balp) \le 1.
\end{equation} 
Together with \eqref{Ldef}, \eqref{Kbounds}, \eqref{Kprod}, \eqref{SstarBound} and \eqref{crux}, this gives
\begin{align} \notag
\int_{\fN^* \cap \ft} |S^*(\balp)  \bK_\pm(\balp)| \d \balp &\ll 
   P^{n-Rd} L(P)^R \sum_{n=0}^\infty (T(P)+n)^{-2}  \\
\label{trivial} &\ll L(P)^R T(P)^{-1} P^{n-Rd} = o(P^{n-Rd}).
\end{align}

Recalling \eqref{NstarDef}, we claim that
\begin{equation} \label{claim}
\fN^* \cap \fM_1  = \{ \balp \in \bR^R: 2|\balp| \le P^{R(d-1)\tet_0 - d}, \: |S(\balp)| > P^{n - R(R+1)d\tet_0} \}.
\end{equation}
It is clear from \eqref{delDef} that if $2|\balp| \le P^{R(d-1)\tet_0 - d}$ and $|S(\balp)| > P^{n - R(R+1)d\tet_0}$ then $\balp \in \fN^* \cap \fM_1$. Conversely, let $\balp \in \fN^* \cap \fM_1$. Then $|S(\balp)| > P^{n - R(R+1)d\tet_0}$. Further, as $\balp \in \fN$ we have $2|\balp - q^{-1} \ba| \le P^{R(d-1)\tet_0 - d}$ for some $q \in \bN$ and $\ba \in \bZ^R$ satisfying $q \le P^{R(d-1)\tet_0}$. Since $\balp \in \fM_1$, the triangle inequality now gives
\[
|q^{-1}\ba| < P^{\del - d} + P^{R(d-1)\tet_0-d} < q^{-1},
\]
so $\ba = \bzero$. Hence $2|\balp| \le P^{R(d-1)\tet_0 - d}$, and we have verified \eqref{claim}.

Put 
\[
\fM = \{ \balp \in \bR^R: 2|\balp| \le P^{R(d-1)\tet_0 - d} \}
\]
and 
\[
\fM_2 = \{ \balp \in \bR^R: 2|\balp| \le P^{R(d-1)\tet_0 - d}, \: |S(\balp)| \le P^{n - R(R+1) d\tet_0} \}.
\]
From \eqref{claim}, we see that $\fM$ is the disjoint union of $\fM_2$ and $\fN^* \cap \fM_1$. 

\begin{lemma}
We have
\begin{equation} \label{MajorSubset}
\fM \subseteq \fR(1,1,1,1,\bzero,\bzero,\bzero,\bzero) \subseteq \fR. 
\end{equation}
\end{lemma}

\begin{proof} Let $\balp \in \fM$, and recall \eqref{suppress}. With $\cX = (1,1,1,1,\bzero,\bzero,\bzero,\bzero)$, the conditions \eqref{bak1cor}, \eqref{qbound}, \eqref{qerror}, \eqref{DE}, and \eqref{further} are plainly met, while the bound \eqref{bak2cor} follows from \eqref{omegadef} and \eqref{delDef}. It therefore remains to show that
\begin{equation} \label{finalmajor0}
|\balp| \ll \max_{|\bj|_1 = j} |\om_\bj| \qquad (j=d,d-1).
\end{equation}
Recall that $d, d-1 \in \cS$. Lemma \ref{SpecLemma} reveals that there exist nonzero integers $D'$ and $E'$, bounded in terms of $\bf$, as well as $\ba'_1, \ba'_2 \in \bZ^R$, satisfying
\begin{equation} \label{finalmajor1}
|D' \balp - \ba'_1| \ll \max_{|\bj|_1 = d} |\om_\bj|, \qquad
|E' \mu \balp - \ba'_2| \ll \max_{|\bj|_1 = d-1} |\om_\bj|.
\end{equation}
By \eqref{bak2cor}, we now have
\[
|D' \balp - \ba'_1|, |E' \mu \balp - \ba'_2| \ll P^{\del - 1}.
\]
Since $\balp \in \fM$, the triangle inequality now gives $|\ba'_1|, |\ba'_2| < 1$, so $\ba'_1 = \ba'_2 = \bzero$. Substituting this information into \eqref{finalmajor1} confirms \eqref{finalmajor0}.
\end{proof}

Now \eqref{Kbounds}, \eqref{Kprod} and \eqref{SstarCompare} yield
\[
\int_{\fM_2} |S^*(\balp)\bK_{\pm}(\balp)| \d \balp \ll P^{R^2(d-1)\tet_0 - Rd}P^{n- R(R+1)d \tet_0} = o(P^{n- Rd}),
\]
so
\begin{align} \notag
\int_{\fN^* \cap \fM_1} S^*(\balp) e(-\balp \cdot \btau) \bK_\pm(\balp) \d \balp 
  &= \int_{\fM} S^*(\balp) e(-\balp \cdot \btau) \bK_\pm(\balp) \d \balp \\
\label{MajorCompare} & \qquad + o(P^{n-Rd}).
\end{align}
By \eqref{Kdef}, we have
\[
K_{\pm}(\alp) = (2 \eta \pm \rho) \cdot \sinc(\pi \alp \rho) \cdot \sinc(\pi \alp(2\eta \pm \rho))
\]
for $\alp \in \bR$. Now \eqref{Tbound}, \eqref{Ldef} and the Taylor expansion of $\sinc(\cdot)$ yield
\[
K_{\pm}(\alp) = 2 \eta + O(L(P)^{-1}) \qquad (|\alp| < P^{-1}).
\]
Substituting this into \eqref{Kprod} gives
\begin{equation} \label{SincTaylor}
\bK_\pm(\balp) = (2\eta)^R + O(L(P)^{-1}) \qquad (\balp \in \fM).
\end{equation}
By \eqref{SstarBound}, \eqref{Fnote} and \eqref{MajorSubset}, we also have
\begin{equation} \label{UpperBound}
\int_\fM |S^*(\balp)| \d \balp \ll P^n \int_{\bR^R} (1+ P^d |\balp|)^{-R-1} \d \balp \ll P^{n-Rd}.
\end{equation}
From \eqref{SincTaylor} and \eqref{UpperBound}, we infer that
\[
\int_\fM S^*(\balp) e(-\balp \cdot \btau) \bK_\pm(\balp) \d \balp
= (2\eta)^R \int_{\fM} S^*(\balp) e(-\balp \cdot \btau)  \d \balp + o(P^{n-Rd}).
\]
Combining this with \eqref{minor}, \eqref{trivial} and \eqref{MajorCompare} yields
\begin{align} \notag
\int_{\fN^*} S^*(\balp) e(-\balp \cdot \btau) \bK_\pm(\balp) \d \balp &= (2\eta)^R \int_\fM S^*(\balp) e(-\balp \cdot \btau)  \d \balp \\
 \label{RemovedK} & \qquad + o(P^{n-Rd}).
\end{align}

Let $\balp \in \fM$. Recall \eqref{Idef} and \eqref{SraDef}. By \eqref{SstarDef} and \eqref{MajorSubset}, we have
\[ 
S^*(\balp) = P^n  e(\balp \cdot \bf( \bmu))
\int_{[-1,1]^n} e\Bigl (\bgam \cdot \bf(\bt) + \sum_{1 \le |\bj|_1 \le d-1} \gam_\bj \bt^\bj \Bigr) \d \bt,
\]
with \eqref{omegadef} and \eqref{WeirdGam}. Using \eqref{WeirdGam} and the change of variables $\by = P \bt$ gives
\[
S^*(\balp) = \int_{[-P,P]^n} e \Bigl( \balp \cdot \bf(\by) + \balp \cdot \bf(\bmu)
+ \sum_{1 \le |\bj|_1 \le d-1} \omega_\bj \by^\bj \Bigr) \d \by.
\]
By \eqref{Taylor} and \eqref{omegadef}, we now have
\[
S^*(\balp) = \int_{[-P,P]^n} e(\balp \cdot \bf(\by + \bmu)) \d \by 
= S_1(\balp) + O(P^{n-1}),
\]
where
\[
S_1(\balp) = \int_{[-P,P]^n} e(\balp \cdot \bf(\bx)) \d \bx.
\]
Hence
\begin{align} \notag
\int_\fM S^*(\balp) e(-\balp \cdot \btau) \d \balp - \int_\fM S_1(\balp) e(-\balp \cdot \btau) \d \balp &\ll P^{n-1 + R^2(d-1)\tet_0 - Rd} \\
\label{S1compare} &= o(P^{n-Rd}).
\end{align}

Note that $S_1(\balp) = P^nI(P^d \balp, \bzero)$. In light of \eqref{kapBound}, the bound \eqref{Ibound} now yields
\[
S_1(\balp) \ll P^n (1+P^d|\balp|)^{- R - 1} \ll P^n \prod_{k \le R} (1+P^d |\alp_k|)^{-1-1/R},
\]
so
\[
\int_{\bR^R \setminus \fM} |S_1(\balp)| \d \balp \ll P^{n - (R-1)d} \int_{P^{R\eps-d}}^\infty (1+P^d \alp)^{-1-1/R} \d \alp \ll P^{n- Rd-\eps}.
\]
In particular
\begin{equation} \label{S1rest}
\int_\fM S_1(\balp) e(-\balp \cdot \btau) \d \balp = \int_{\bR^R} S_1(\balp) e(-\balp \cdot \btau) \d \balp + o(P^{n-Rd}).
\end{equation}

To apply \cite[Lemma 5.3]{Bir1962} directly, we need to work with a box of side length less than 1. Changing variables with $\bx = 3P \bu$ and $\bz = (3P)^d \balp$ shows that
\begin{align*}
\int_{\bR^R} S_1(\balp) e(- \balp \cdot \btau) \d \balp &=  \int_{\bR^R} \int_{[-P,P]^n} e(\balp \cdot \bf(\bx)) e(-\balp \cdot \btau) \d \bx  \d \balp \\
&= (3P)^{n - Rd} \int_{\bR^R} \cI(\bz)  e( - (3P)^{-d} \btau \cdot \bz)\d \bz,
\end{align*}
where 
\[
\cI(\bz) = \int_{[-1/3,1/3]^n} e(\bz \cdot \bf(\bu)) \d \bu.
\]
Now \cite[Lemma 5.3]{Bir1962} gives 
\[
\int_{\bR^R} S_1(\balp) e(- \balp \cdot \btau) \d \balp = (3P)^{n - Rd} \Bigl( \int_{\bR^R} \cI(\bz) \d \bz + o(1) \Bigr)
\]
as $P \to \infty$. Moreover, changing variables yields
\[
\int_{\bR^R} \cI(\bz) \d \bz = \int_{\bR^R} \int_{[-1/3,1/3]^n} e(\bz \cdot \bf(\bu)) \d \bu \d \bz = 3^{Rd-n} c,
\]
where we recall \eqref{cdef}. Hence
\begin{equation} \label{S1eval}
\int_{\bR^R} S_1(\balp) e(- \balp \cdot \btau) \d \balp = cP^{n-Rd} + o(P^{n-Rd}).
\end{equation}

Combining \eqref{S1compare}, \eqref{S1rest} and \eqref{S1eval} gives
\[
\int_\fM S^*(\balp) e(-\balp \cdot \btau) \d \balp = cP^{n-Rd} + o(P^{n-Rd}).
\]
Substituting this into \eqref{RemovedK} yields \eqref{goal3}, confirming the desired asymptotic formula \eqref{asymp}.

\section{The singular integral} 
\label{SingularIntegral}

Schmidt \cite[\S 3]{Sch1985} gives the following geometric definition of the real density $c$. For $L > 0$ and $\xi \in \bR$, let
\[
\lam_L(\xi) = L \cdot \max(0,1 - L|\xi|).
\]
For $\bxi \in \bR^R$, put
\[
\Lam_L(\bxi) = \prod_{k \le R} \lam_L(\xi_k).
\]
Set
\[
I_L(\bf) = \int_{[-1,1]^n} \Lam_L(\bf(\bt)) \d \bt,
\]
and define
\begin{equation} \label{cSchmidt}
c = \lim_{L \to \infty} I_L(\bf)
\end{equation}
whenever the limit exists. Schmidt explains in \cite[\S 11]{Sch1982} and \cite[\S 3]{Sch1985} that the limit does exist, and that this definition is equivalent to Birch's analytic definition \eqref{cdef}.

The expression on the right hand side of \eqref{cdef} arose naturally in our proof of \eqref{asymp}. It is well defined, by \cite[Lemma 5.3]{Bir1962} and a change of variables (here Birch uses a box of side length less than 1). One can verify the final statement of Theorem \ref{MainThm} from \eqref{cSchmidt} by mimicking \cite[\S 4]{Sch1982}; one uses the implicit function theorem to construct a region of measure $\gg L^{-R}$ on which $|\bf(\bt)| < (2L)^{-1}$. Birch instead invokes the Fourier integral theorem to show from \eqref{cdef} that $c > 0$ whenever $\bf = \bzero$ has a nonsingular real solution (see \cite[\S6]{Bir1962}). 

This discussion concludes the proof of Theorem \ref{MainThm}.

\section{An alternative approach}
\label{SchmidtApproach}

In this section we establish Theorem \ref{TheoremTwo}. The crux is a suitable analogue of Lemma \ref{Birch43}, and we shall deduce such an analogue from the work of Schmidt \cite{Sch1985}. Let $g$ be as defined in \cite[\S 10]{Sch1985}, and put
\begin{equation*}
\kap' = \frac g{R(d-1)2^{d-1}}.
\end{equation*}
The quantity $\kap'$ shall play the r\^ole played by $\kap$ in the proof of Theorem \ref{MainThm}. We note at once that coupling \eqref{hhyp} with the corollary to \cite[Proposition III]{Sch1985} yields
\[
\kap' > R+1,
\]
in analogy with \eqref{kapBound}.

We begin with an analogue of \cite[Lemma 2.5]{Bir1962}.

\begin{lemma} \label{Bir25}
Let $0 < \tet \le 1$ and $k > 0$. Then at least one of the following holds.
\begin{enumerate}[(i)]
\item We have
\[
g(\balp, \bom_\diam) \ll P^{n-k}.
\]
\item There exist integers $q, a_1, \ldots, a_R$ satisfying \eqref{qa} and \eqref{FirstApprox}.
\item We have
\[
g \le 2^{d-1}k/\tet.
\]
\end{enumerate}
The same is true if we replace $g(\balp, \bome_\diam)$ by
\[
\sum_{1 \le x_1, \ldots, x_n \le P}  e\Bigl( \balp \cdot \bf(\bx) + \sum_{1 \le |\bj|_1 \le d-1} \omega_\bj \bx^\bj \Bigr).
\]
\end{lemma}

\begin{proof} We may imitate the Birch's proof of \cite[Lemma 2.5]{Bir1962}. As in \S \ref{BirchType}, the lower order terms have no bearing on the proof. Our second assertion follows in the same way as our first.
\end{proof}

This implies the following analogue of Lemma \ref{Birch43}.

\begin{lemma} \label{Birch43alt}
Let $0 < \tet \le 1$. Suppose 
\begin{equation} \label{gbigalt}
|g(\balp, \bome_\diam)| > P^{n- R(d-1)\kap' \tet + \eps}.
\end{equation}
Then there exist integers $q, a_1, \ldots, a_R$ satisfying \eqref{qa} and \eqref{FirstApprox}. In particular, if $|S(\balp)| >  P^{n- R(d-1)\kap' \tet + \eps}$ then there exist $q \in \bN$ and $\ba \in \bZ^R$ satisfying \eqref{qa} and \eqref{FirstApprox}. We may replace $g(\balp, \bome_\diam)$ by
\[
\sum_{1 \le x_1, \ldots, x_n \le P}  e\Bigl( \balp \cdot \bf(\bx) + \sum_{1 \le |\bj|_1 \le d-1} \omega_\bj \bx^\bj \Bigr),
\]
and the same conclusions hold. 
\end{lemma}

\begin{proof}
Choosing $k = R(d-1) \kap' \tet - \eps$ in Lemma \ref{Bir25} ensures that (iii) is impossible, reducing us to two possibilities. We have removed the implied constant from \eqref{gbigalt} by redefining $\eps$ and recalling that $P$ is large. Our second claim follows from our first, by \eqref{Sg}.
\end{proof}

Using Lemma \ref{Birch43alt} instead of Lemma \ref{Birch43}, we can then follow the proof of Theorem \ref{MainThm}, with minimal changes. Corollary \ref{Birch43cor} follows with $\kap'$ in place of $\kap$. Similarly, Lemmas \ref{IboundLemma} and \ref{Sra} follow in the same way, but with $\kap'$ in place of $\kap$. Finally, it is important to note that we still have $d, d-1 \in \cS$. As explained in the introduction, our assumption that the $(1,\ldots,1) \cdot \nabla f_k$ are linearly independent implies that $d-1 \in \cS$. This assumption also implies that $d \in \cS$, in view of \eqref{top}. This completes the proof of Theorem \ref{TheoremTwo}.

The quantity $\Phi(d)$ dominates the quantity $\varphi(d)$ in \cite[Proposition III$_C$]{Sch1985}. If we read \cite{Sch1985} more closely, we find that we can replace $\Phi(d)$ by
\[
\max(\eta_{d-2}, 2^{d-2} - 1),
\]
where $\eta_0 = 1$ and
\[
\eta_m = \sum_{q=1}^m \sum_{\substack{u_1 + \ldots + u_m = q \\ u_i > 0}} \frac {m!} {u_1 ! \cdots u_q!}
\qquad (m \in \bN).
\]

\bibliographystyle{amsbracket}
\providecommand{\bysame}{\leavevmode\hbox to3em{\hrulefill}\thinspace}

\end{document}